\documentclass[a4paper, 10pt]{amsart}

\usepackage[english]{babel}
\usepackage{amsmath}
\usepackage[all]{xy}
\usepackage{mathrsfs}  

\usepackage{hyperref}

\newcommand{\Q}{\mathbf{Q}}
\newcommand{\C}{\mathbf{C}}
\newcommand{\N}{\mathbf{N}}
\newcommand{\Z}{\mathbf{Z}}

\renewcommand{\O}{\mathscr{O}}
\newcommand{\R}{\mathbf{R}}
\newcommand{\Hyp}{\mathbb{H}}

\newcommand{\Sm}{\mathsf{Sm}}
\newcommand{\Smf}{\mathsf{Smf}}
\newcommand{\sS}{\mathsf{S}}
\newcommand{\Ch}{\mathsf{Ch}}
\newcommand{\Sch}{\mathsf{Sch}}

\newcommand{\op}{\mathrm{op}}

\DeclareMathOperator{\Hom}{Hom}
\DeclareMathOperator{\GL}{GL}
\DeclareMathOperator{\Gal}{Gal}
\DeclareMathOperator{\id}{id}
\DeclareMathOperator{\pr}{pr}
\newcommand{\rel}{\mathrm{rel}}
\DeclareMathOperator{\Spec}{Spec}
\DeclareMathOperator{\Sp}{Sp}
\DeclareMathOperator{\Spwf}{Spwf}
\DeclareMathOperator{\Tot}{Tot}
\DeclareMathOperator{\reg}{reg}
\newcommand{\Fil}{F}
\DeclareMathOperator{\ch}{ch}
\DeclareMathOperator{\Cone}{Cone}
\DeclareMathOperator{\MF}{mf}
\renewcommand{\top}{\mathrm{top}}
\newcommand{\rig}{\mathrm{rig}}
\newcommand{\syn}{\mathrm{syn}}
\newcommand{\dR}{\mathrm{dR}}
\newcommand{\et}{\mathrm{\acute{e}t}}
\DeclareMathOperator{\spec}{sp}
\DeclareMathOperator{\Gd}{Gd}
\DeclareMathOperator{\diag}{diag}
\DeclareMathOperator{\sgn}{sgn}

\DeclareMathOperator*{\holim}{holim}

\newcommand{\del}{\partial}
\newcommand{\ol}[1]{\overline{#1}}
\newcommand{\ul}[1]{\underline{#1}}

\newcommand{\<}{\langle}
\renewcommand{\>}{\rangle}

\renewcommand{\dot}{\bullet}

\newcommand{\qis}{\simeq}

\newtheorem{thm}[equation]{Theorem}
\newtheorem*{thm*}{Theorem}
\newtheorem{cor}[equation]{Corollary}
\newtheorem*{cor*}{Corollary}
\newtheorem{lemma}[equation]{Lemma}
\newtheorem{prop}[equation]{Proposition}

\theoremstyle{definition}
\newtheorem{dfn}[equation]{Definition}
\newtheorem*{dfn*}{Definition}

\theoremstyle{remark}
\newtheorem{rem}[equation]{Remark}
\newtheorem*{rem*}{Remark}
\newtheorem{rems}[equation]{Remarks}
\newtheorem*{ex*}{Example}

\newtheorem{Nr}[equation]{}

\numberwithin{equation}{section}

\title[Karoubi's relative Chern character and the syntomic regulator]{Karoubi's relative Chern character, the rigid syntomic regulator, and the Bloch-Kato exponential map}
\author{Georg Tamme}
\address{Fakult\"at f\"ur Mathematik, Universit\"at Regensburg, 93040 Regensburg, Germany}
\email{georg.tamme@ur.de}
\urladdr{mathematik.uni-r.de/tamme}

\date{\today}

\begin{document}
\begin{abstract}
We construct a variant of Karoubi's relative Chern character for smooth, separated schemes over the ring of integers in a $p$-adic field and prove a comparison with the rigid syntomic regulator. For smooth projective schemes we further relate the relative Chern character to the \'etale $p$-adic regulator via the Bloch-Kato exponential map. This reproves a result of Huber and Kings for the spectrum of the ring of integers and generalizes it to all smooth projective schemes as above.
\end{abstract}

\maketitle

\tableofcontents

\section*{Introduction}

If $A$ is a Banach algebra one can view it as an abstract ring and consider its algebraic K-theory, or one can take the topology into account and then consider its topological K-theory. There is a natural map from the former to the latter and so one can form the homotopy fibre of this map giving the \emph{relative} K-theory. Karoubi's relative Chern character \cite{KarCR, Kar87, CK} is a homomorphism
\[
K_{i}^{\rel}(A) \to HC_{i-1}(A)
\]
mapping relative K-theory to continuous cyclic homology. It was an idea of Karoubi \cite{KaroubiConnexions, KarCR} that the relative Chern character could be used for the construction of regulators. In accordance with this idea Hamida  \cite{HamidaBorel} established, 
for $A=\mathbf C$ the field of complex numbers, a precise relation between the relative Chern character and the Borel regulator \cite{Borel1}
\[
K_{2n-1}(\mathbf C) \to \mathbf C.
\]
Karoubi's construction works equally well in the case of ultrametric Banach algebras and building on previous work by Hamida \cite{HamidaCR} we proved the $p$-adic analogue of the above result in \cite{TammeBorel}, giving the precise relation of the relative Chern character with the $p$-adic Borel regulator introduced by Huber and Kings \cite{HK}.

Changing perspective, let $X$ be a smooth variety over $\mathbf C$. Again we have the algebraic K-theory of $X$ but we can also consider the topological K-theory of the complex manifold associated to $X$. It is natural to ask for a generalization of the previous results to this situation. Here the analogue of the cyclic homology groups are the quotients of the de Rham cohomology by the Hodge filtration and Borel's regulator is replaced by Beilinson's regulator mapping algebraic K-theory to Deligne-Beilinson cohomology. In this setup we proved a comparison between the relative Chern character and Beilinson's regulator in \cite{TammeBeil}.

It is the goal of the present paper to prove the $p$-adic analogue of this result. Let $R$ be a complete discrete valuation ring with field of fractions $K$ of  characteristic $0$ with perfect residue field $k$ of characteristic $p$ and consider a smooth $R$-scheme $X$. The $p$-adic analogue of Beilinson's regulator is the rigid syntomic regulator, i.e.~the Chern character with values in rigid syntomic cohomology, introduced by Gros \cite{Gros} and developed systematically by Besser \cite{Besser}. We introduce topological and hence relative K-theory of $X$, and relative cohomology groups $H_{\rel}^{*}(X,n)$ mapping naturally to the rigid syntomic cohomology groups. These are the target of the relative Chern character in the $p$-adic situation. Our main result is
\begin{thm*}
Let $X$ be a smooth $R$-scheme and $i >0$. The diagram
\[
\xymatrix{
K_{i}^{\rel}(X) \ar[d]_{\ch_{n,i}^{\rel}} \ar[r] & K_{i}(X) \ar[d]^{\ch_{n,i}^{\syn}}\\
H^{2n-i}_{\rel}(X, n) \ar[r] & H^{2n-i}_{\syn}(X,n)
}
\]
commutes. 
\end{thm*}
If $X$ is proper the lower horizontal map is in fact an isomorphism and both groups are given by
$H^{2n-i-1}_{\dR}(X_{K}/K)/\Fil^{n}H^{2n-i-1}_{\dR}(X_{K}/K)$
which, in turn, is naturally isomorphic to the weight $n-1$ part $HC_{i-1}^{(n-1)}(X_K)$ in the
$\lambda$-decomposition of the cyclic homology of $X_K/K$ \cite{WeibelHodgeCyclic}.
Moreover, for projective $X$ and finite $k$, Parshin's conjecture would imply that also the upper horizontal map is rationally an isomorphism.

One of the possible advantages of this approach to the syntomic regulator is that Karoubi's constructions give quite explicit formulas. For instance, in the case $X=\Spec(R)$, these have been used in the comparison of Karoubi's regulator with the $p$-adic Borel regulator \cite{TammeBorel} and in computer calculations by Choo and Snaith \cite{ChooSnaith}.

Another motivation to study the relative Chern character and its relation to the syntomic regulator goes back to an idea of Besser.
In contrast to the Beilinson regulator or Soul\'e's \'etale $p$-adic regulator, the rigid syntomic regulator explicitly depends on the choice of the local model $X/R$ of the generic fibre $X_{K}/K$. In computations (e.g.~\cite{BesserK1surface,Besser-deJeu}) this leads to integrality assumptions one would like to remove. Besser proposed the use of Karoubi's relative Chern character in order to obtain a model independent replacement for the syntomic regulator. In fact, one can define topological and hence relative K-theory of $K$-schemes using the associated rigid space, and the techniques of this paper give a relative Chern character $K_{i}^{\rel}(X_{K}) \to H^{2n-i-1}_{\dR}(X_{K}/K)/\Fil^{n}$. If $X$ is a smooth, proper $R$-scheme there is a natural map $K_{*}^{\rel}(X) \to K_{*}^{\rel}(X_{K})$ and the relative Chern character for $X$ factors through the relative Chern character of $X_{K}$. 

In particular, if one assumes Parshin's conjecture, the relative Chern character would give a description of the syntomic regulator for $X$ smooth, projective over $R$ with finite residue field solely in terms of the generic fibre $X_{K}$. In general, a good understanding of topological K-theory is still missing.

From the Theorem we also get the following corollary (see \ref{apr1201}): 
\begin{cor*}
Assume that $k$ is finite and $X/R$ is smooth and projective. Then
\[
\xymatrix@C+0.5cm{
K_{i}^{\rel}(X) \ar[d]^{\ch_{n,i}^{\rel}} \ar[r] & K_{i}(X) \ar[d]^{r_{p}}\\
H^{2n-i-1}_{\dR}(X_{K}/K)/\Fil^{n} \ar[r]^-{\exp} & H^{1}(K, H^{2n-i-1}_{\et}(X_{\ol K}, \Q_{p}(n)))
}
\]
commutes.
\end{cor*}
Here $r_{p}$ is the \'etale $p$-adic regulator and $\exp$ is the Bloch-Kato exponential map of the $p$-adic $\mathrm{Gal}(\ol K/K)$-representation $H^{2n-i-1}_{\et}(X_{\ol K}, \Q_{p}(n))$.
This corollary may be seen as a generalization of the main result of Huber-Kings \cite[Theorem 1.3.2]{HK} which is the case $X=\Spec(R)$ and amounts to the commutativity of
\[
\xymatrix{
K_{2n-1}(R) \ar[r]^-{r_{p}} \ar[dr]_-{p\text{-adic Borel}} & H^{1}(K, \Q_{p}(n))\\
& K=D_{\dR}(\Q_{p}(n)) \ar[u]_{\exp}
}
\]
to all smooth, projective $R$-schemes (see our Corollary \ref{cor:HK}). 

\smallskip

A result related to ours is proven by Chiarellotto, Ciccioni, and Mazzari in \cite{Mazzari}. They provide an alternative construction of the rigid syntomic regulator in terms of higher Chow groups and syntomic cycle classes. A key step in their construction is the compatibility of the de Rham and rigid cycle classes under the specialization map from de Rham cohomology of the generic fibre to rigid cohomology of the special fibre, whereas our result in some sense rests on the compatibility of topological and rigid Chern classes.

For an interpretation of the relative cohomology groups, introduced here, in stable homotopy theory of schemes we refer the reader to \cite{DegliseMazzari}.

\smallskip

Let us describe the contents of this paper in more detail. 
Karoubi's original construction of the relative Chern character uses integration of certain $p$-adic differential forms over standard simplices. A reformulation of this construction 
is given in Section \ref{sec:Karoubi}.
The key ingredient that enables us to compare the relative Chern character and the rigid syntomic regulator is a new description of the former in Section \ref{sec:Construction} that is similar to the construction of Chern class maps on higher K-groups by Beilinson \cite{Beilinson}, Huber \cite{Huber}, and Besser \cite{Besser}. This is made possible by the functorial complexes of Besser \cite{Besser} and Chiarellotto-Ciccioni-Mazzari \cite{Mazzari}. Their construction is recalled in Section \ref{sec:Functorial} with some simplifications coming from the systematic use of Gro{\ss}e-Kl\"onne's dagger spaces \cite{GK, GKRigid}. The necessary background from rigid resp. ``dagger'' geometry is collected in Section \ref{sec:Rigid}. In Section \ref{sec:K} we recall Karoubi and Villamayor's definition of topological K-theory for ultrametric Banach rings \cite{KV, Calvo} and extend it to smooth, separated $R$-schemes.
The necessary comparison between the two constructions of the relative Chern character is proven in Theorem \ref{thm:KaroubisRelChern}.
The main comparison theorem (Theorem \ref{thm:Comparison}) then follows rather formally. 
Applications are given in Subsection \ref{ssec:applications}.

\subsection*{Acknowledgments}
The results presented here emerged from my thesis \cite{TammeDiss}. It is a pleasure to thank my
advisor Guido Kings for his interest and constant support.
I would like to thank the California Institute of Technology, where most of this work was done, and
especially Matthias Flach for their hospitality and the Deutsche Forschungsgemeinschaft for
financial support.
Furthermore, I would like to thank Nicola Mazzari for interesting discussions about this work and the referee for several useful remarks.

\subsection{Notation}\label{sec:notation}
For $p \in \N_{0}=\{0,1,2,\dots\}$ we denote by $[p]$ the finite set $\{0,\dots, p\}$ with its natural order. The category of finite ordered sets with monotone maps is the \emph{simplicial category} $\Delta$.  The unique injective map $[p-1] \to [p]$ that does not hit $i$ is denoted by $\del^{i}$. Similarly, $s^{i}\colon [p+1] \to [p]$ is the unique surjective map such that $s^{i}(i)=s^{i}(i+1)$. These induce morphisms $\del_{i}, s_{i}$ (resp. $\del^{i}, s^{i}$) on every (co)simplicial object called (co)face and (co)degeneracy morphisms, respectively. 

For a group object $G$ we define the simplicial objects $E_{\dot}G$ and $B_{\dot}G$ by $E_{p}G=G^{\times (p+1)}$ and $B_{p}G = G^{\times p}$ with the usual faces and degeneracies (see e.g.~\cite[0.2]{HK}).

\section{Preliminaries on rigid geometry}
\label{sec:Rigid}

In the definition of rigid cohomology one usually works with rigid analytic spaces and their de Rham cohomology. However, as de Rham cohomology is not  well behaved for rigid spaces one has to introduce some overconvergence condition. An elegant approach to do this is to replace rigid spaces by Gro{\ss}e-Kl\"onne's dagger spaces \cite{GK, GKRigid}. We recall some basic definitions and facts which will be needed in the rest of the paper.

Let $R$ be a complete discrete valuation ring with field of fractions $K$ of characteristic 0 and residue field $k$ of characteristic $p>0$. Fix an absolute value $|\,.\,|$ on $K$.

\begin{Nr}[Rings]\label{Nr:DaggerRings}
The $K$-algebra of \emph{overconvergent} power series in $n$ variables $\ul x = (x_{1}, \dots, x_{n})$ is   
\(
K\<\ul x\>^{\dag} := \left\{\sum a_{\nu}\ul x^{\nu} \big|a_{\nu}\in K, \exists\rho>1:|a_{\nu}|\rho^{|\nu|}\xrightarrow{|\nu|\to\infty} 0\right\}.
\)
We denote by $R\<\ul x\>^{\dag}$ the sub-$R$-algebra of overconvergent power series with $R$-coefficients. A $K$- (resp. $R$-)\emph{dagger algebra} is a quotient of some $K\<\ul x\>^{\dag}$ (resp. $R\<\ul x\>^{\dag}$).

The algebra of overconvergent power series carries the Gau{\ss} norm $\left|\sum_{\nu} a_{\nu}\ul x^{\nu}\right| =\sup_{\nu}|a_{\nu}|$. Its completion with respect to this norm is the Tate algebra of convergent power series $K\<\ul x\> =\{\sum_{\nu}a_{\nu}x^{\nu} \big| |a_{\nu}| \xrightarrow{|\nu|\to\infty} 0\}$. Similarly, the completion of $R\<x\>^{\dag}$ is $R\<x\>$. Quotients of these are called $K$- (resp. $R$-)\emph{affinoid}. These are Banach algebras. Up to equivalence the quotient norm on a dagger or affinoid algebra does not depend on the chosen representation as a quotient of an algebra of overconvergent, respectively, convergent power series.

To any $R$-algebra $A$ one can associate its \emph{weak completion} $A^\dag$ (cf.~\cite[Dfn.~1.1]{MW}).
If $A$ is an $R$-algebra of finite type and $A\cong R[\ul x]/I$ is a presentation,
then there is an induced isomorphism  $A^{\dag} \cong R\<\ul x\>^{\dag}/IR\<\ul x\>^{\dag}$. In particular, $A^\dag$ is an $R$-dagger algebra.
Similarly, there is a weak completion for normed $K$-algebras.

The categories of $R$- resp. $K$-dagger and affinoid algebras admit tensor products. 
E.g., if $A$ and $B$ are $R$-dagger algebras, their tensor product is $A\otimes_R^\dag B := (A\otimes_R B)^\dag$. Given presentations
 $A\cong R\<\ul x\>^{\dag}/I$ and $B\cong R\<\ul y\>^{\dag}/J$, there is a natural isomorphism $A \otimes_{R}^{\dag} B \cong R\<\ul x, \ul y\>^{\dag}/(I+J)$.
\end{Nr}

\begin{Nr}[Spaces]\label{Nr:DaggerSpaces}
We only sketch the main points here, referring the reader to \cite{GK, GKRigid} for details. Similarly as one defines rigid analytic spaces that are locally isomorphic to max-spectra of $K$-affinoid algebras with a certain Grothendieck topology, \emph{dagger spaces} are defined by taking the max-spectra $\Sp(A)$ of $K$-dagger algebras $A$ as building blocks \cite[2.12]{GKRigid}. For any dagger space $X$ one has an associated rigid space $X^{\rig}$ (``completion of the structure sheaf'') and a natural map of G-ringed spaces $X^{\rig} \xrightarrow{u} X$ which is an isomorphism on the underlying G-topological spaces \cite[2.19]{GKRigid}.

There exists a \emph{dagger analytification functor} $X\mapsto X^{\dag}$ from the category of $K$-schemes of finite type to the category of $K$-dagger spaces. There is a natural morphism of locally G-ringed spaces $X^{\dag} \xrightarrow{\iota} X$ which is final for morphisms from dagger spaces to $X$ \cite[3.3]{GKRigid}.

We also need the notion of \emph{weak formal schemes} (\cite[Ch. 3]{GK} and originally \cite{Mer}). Let $A$ be an $R$-dagger algebra and $\ol A = A\otimes_{R} k$. Then $D(\ol f) \mapsto A\<\frac{1}{f}\>^{\dag}$ where $f\in A$ lifts $\ol f\in \ol A$ defines a sheaf of local rings on the topological space $\Spec(\ol A)$. The corresponding locally ringed space is the \emph{weak formal $R$-scheme} $\Spwf(A)$. A general weak formal $R$-scheme is a locally ringed space that is locally isomorphic to some $\Spwf(A)$. 

Sending $\Spwf(A)$ for an $R$-dagger algebra $A$ to $\Sp(A\otimes_{R} K)$ induces the \emph{generic fibre} functor $\mathscr X \mapsto \mathscr X_{K}$ from weak formal $R$-schemes to $K$-dagger spaces and there is a natural \emph{specialization map} $\spec\colon \mathscr X_{K} \to \mathscr X$.

Taking the weak completion of finitely generated $R$-algebras induces the functor $X \mapsto \widehat X$ from $R$-schemes of finite type to weak formal $R$-schemes. There is a natural morphism of dagger spaces from the generic fibre of the weak completion of $X$, $\widehat X_{K}$, to the dagger analytification $X_{K}^{\dag}$ of the generic fibre $X_{K}$ of $X$. This is an open immersion if $X$ is separated and an isomorphism if $X$ is proper over $R$ (cf. \cite[Proposition 0.3.5]{BerCohomRig}). For example, if $X=\mathbb A^{1}_{R}$ then $\widehat X_{K}=\Sp(K\<x\>^{\dag})$ is the closed ball of radius 1 in $(\mathbb A^{1}_{K})^{\dag}$.

\end{Nr}

\section{Preliminaries on K-theory}
\label{sec:K}

\begin{Nr}
Let $A$ be an ultrametric Banach ring with norm $|\,.\,|$, e.g. an affinoid algebra with a fixed norm. In \cite{KV} Karoubi and Villamayor introduce K-groups of $A$ that we will denote by $K_{\top}^{-i}(A), i\geq 1$, which were further studied by Calvo \cite{Calvo}. A convenient way to define them is the following: Set 
\[
A\<x_{0}, \dots, x_{n}\> := \left\{\sum a_{\nu}\ul x^{\nu} \big|a_{\nu}\in A, |a_{\nu}|\xrightarrow{|\nu|\to\infty} 0\right\}
\]
and $A\<\Delta^n\> := A\<x_{0}, \dots, x_{n}\>/(\sum_{i} x_{i} -1)$. Then $[n]\mapsto A\<\Delta^n\>$ becomes a simplicial ring in a natural way and hence $B_{\dot}\GL(A\<\Delta^{\dot}\>)$ is a bisimplicial set (cf. \ref{sec:notation}). 
For any bisimplicial set $S_{\dot\dot}$, we denote by $\pi_*(S_{\dot\dot})$ the homotopy groups of the underlying diagonal simplicial set.
We define
\[
K^{-i}_{\top}(A) := \pi_{i}(B_{\dot}\GL(A\<\Delta^{\dot}\>)), \quad i\geq 1.
\]
That this definition coincides with the original one in \cite{KV, Calvo} follows from an argument of Anderson 
 \cite[Prop. 7.3]{TammeDiss}.

For our purposes it is important to know that one can compute the K-theory of affinoid algebras using dagger algebras. More precisely, let $R$ be a complete discrete valuation ring with field of fractions $K$ of characteristic 0 and residue field $k$ of characteristic $p>0$ as before.
We define the simplicial ring $R\<\Delta^{\dot}\>^{\dag}$ by $R\<\Delta^{n}\>^{\dag} = R\<x_{0}, \dots, x_{n}\>^{\dag}/(\sum_{i} x_{i} -1)$ with the obvious structure maps. For any $R$-dagger algebra $A$ we set $A\<\Delta^{\dot}\>^{\dag}=A\otimes_{R}^{\dag} R\<\Delta^{\dot}\>^{\dag}$ (cf. \ref{Nr:DaggerRings}) and define topological K-groups by
\[
K^{-i}_{\top}(A) := \pi_{i}(B_{\dot}\GL(A\<\Delta^{\dot}\>^{\dag})), \quad i\geq 1.
\]
Using Calvo's techniques we have shown in \cite[Prop. 7.5]{TammeDiss} that these agree with the Karoubi-Villamayor K-groups of the completion $\widehat A$ of $A$:
\[
K^{-i}_{\top}(A) \cong K^{-i}_{\top}(\widehat A), \quad i\geq 1.
\]
\end{Nr}

Now let $X = \Spec(A)$ be an affine scheme of finite type over $R$. 
\begin{dfn} We define the topological K-groups of $X$ to be the topological K-groups of the (weak) completion of $A$:
\[
K^{-i}_{\top}(X) := K^{-i}_{\top}(A^{\dag}) = K^{-i}_{\top}(\widehat A), \quad i\geq 1.
\]
\end{dfn}

\begin{rems}\label{nov0201}
(i) Note the similarity with topological complex K-theory: If $X$ is a smooth separated scheme of finite type over $\C$, and $A_{\dot}$ denotes the simplicial ring of smooth functions $X(\C) \times \Delta^{\dot} \to \C$, then $\pi_{i}(B_{\dot}\GL(A_{\dot})) \cong K^{-i}_{\top}(X(\C))$ is the connective complex K-theory of the manifold $X(\C)$.

\noindent (ii) Let $\pi\in R$ be a uniformizer of $R$ and $B$ any $R$-dagger or affinoid algebra. Then $(\pi) \subseteq B$ is topologically nilpotent. 
Calvo \cite{Calvo} proved that the reduction $B \to B/(\pi)$ induces an isomorphism $K^{-i}_{\top}(B) \xrightarrow{\cong} K^{-i}_{\top}(B/(\pi))$. This last group is the Karoubi-Villamayor K-theory of $B/(\pi)$. In particular, if $B/(\pi)$ is regular, this coincides with the Quillen K-theory: $K^{-i}_{\top}(B) \xrightarrow{\cong} K_{i}(B/(\pi))$ \cite[3.14]{Gersten}.

\noindent (iii) If, in the situation of the definition, $\pi$ is invertible on $X = \Spec(A)$, i.e., the special fibre $X_{k}$ is empty, then the completion $\widehat A$ is the zero ring and the topological K-theory of $X$ vanishes.
\end{rems}

\begin{Nr}[Connection with algebraic K-theory] Recall that for any ring $A$, the Karoubi-Villamayor K-groups \cite{KV} can be defined as
\[
 KV_i(A) = \pi_i\left(B_\dot\GL(A[\Delta^\dot])\right), \quad i\geq 1,
\]
where $A[\Delta^\dot]$ is the simplicial ring with $A[\Delta^n] = A[x_0, \dots, x_n]/(\sum_i x_i -1)$.
There is a natural map from the Quillen K-group $K_i(A)$ to $KV_i(A)$ which is an isomorphism when $A$ is regular. Since we are only interested in the case of regular rings, we will in the following identify $K_i(A) = KV_i(A)$.

Consider a smooth affine $R$-scheme $X = \Spec(A)$ as above. There is a natural map $A[\Delta^\dot] \to A^\dag\<\Delta^\dot\>^\dag$. We define the bisimplicial set
\begin{equation}\label{nov0904}
F(X) := F(A) := B_\dot\GL\left(A[\Delta^\dot]\right) \times_{B_\dot\GL\left(A^\dag\<\Delta^\dot\>^\dag\right)} E_\dot\GL\left(A^\dag\<\Delta^\dot\>^\dag\right)
\end{equation}
and the \emph{relative K-groups of $X$}
\[
 K_{i}^{\rel}(X) := \pi_{i} F(A),\quad i>0.
\]
We will also need the following finite level variant of $F(A)$:
\begin{equation}\label{dez1202}
F_{r}(A) := B_\dot\GL_{r}\left(A[\Delta^\dot]\right) \times_{B_\dot\GL_{r}\left(A^\dag\<\Delta^\dot\>^\dag\right)} E_\dot\GL_{r}\left(A^\dag\<\Delta^\dot\>^\dag\right)
\end{equation}
so that $F(A) = \varinjlim_{r} F_{r}(A)$ and $K_{i}^{\rel}(X) = \varinjlim_{r} \pi_{i}(F_{r}(A))$.
Since the projection $E_\dot\GL(A^\dag\<\Delta^\dot\>^\dag) \to B_\dot\GL(A^\dag\<\Delta^\dot\>^\dag)$ is a Kan fibration on the diagonal simplicial sets and since $E_\dot\GL(A^\dag\<\Delta^\dot\>^\dag)$ is contractible  we get:
\end{Nr}
\begin{lemma}\label{lem:KRel}
There are long exact sequences
\[
\dots \to K_{i}^{\rel}(X) \to K_{i}(X) \to K^{-i}_{\top}(X) \to K_{i-1}^{\rel}(X) \to \dots 
\]
\end{lemma}
In the following, we will use the notation $\GL(X):= \GL(\Gamma(X,\O_X))$ where $X$ can be a scheme, a dagger space, or a weak formal scheme.

We extend the definition of topological and relative K-theory to smooth, separated $R$-schemes $X$ of finite type as follows. 
Write $\Delta^\dot_R := \Spec(R[\Delta^\dot])$. 
Since K-theory for regular schemes satisfies Zariski descent
we have  isomorphisms 
\begin{equation}\label{eq:AlgKDescent}
K_{i}(X) \cong \varinjlim_{U_{\dot}\to X} \pi_{i}\holim_{[q]\in\Delta}(B_{\dot}\GL(U_{q}\times_R\Delta^\dot_R)), \quad i\geq 1
\end{equation}
where $U_{\dot}\to X$ runs through all finite affine open coverings of $X$ viewed as simplicial schemes (cf. \cite[Prop. 18.1.5]{Huber}).
In analogy to \eqref{eq:AlgKDescent} we define
\begin{gather*}
K^{-i}_{\top}(X) := \varinjlim_{U_{\dot}\to X} \pi_{i} \holim_{[q]\in\Delta} B_{\dot}\GL(\widehat U_{q} \times_R \widehat\Delta^{\dot}_R), \quad i\geq 1, \text{ and}\\
K_{i}^{\rel}(X):= \varinjlim_{U_{\dot}\to X} \pi_{i} \holim_{[q]\in\Delta} F(U_{q}).
\end{gather*}
Here $\widehat \Delta^p_R$ is the weak completion of the algebraic standard simplex
 $\Delta^{p}_R$  so that if $U=\Spec(A)$ then $\widehat U \times_R \widehat\Delta^{p}_R = \Spwf(A^{\dag} \otimes_{R}^{\dag} R\<\Delta^p\>^{\dag})$ (see \ref{Nr:DaggerRings}).
These groups are contravariantly functorial in $X$.

\begin{lemma}\label{lemma:RelKSmoothSeparated}
\begin{enumerate}
\item If $X=\Spec(A)$ is affine and smooth over $R$ these definitions coincide with the earlier ones.
\item If $X$ is a smooth, separated $R$-scheme of finite type, there is an isomorphism $K^{-i}_\top(X) \cong K_i(X_k)$.
\item There are long exact sequences 
\[
\dots \to K_{i}^{\rel}(X) \to K_{i}(X) \to K^{-i}_{\top}(X) \to K_{i-1}^{\rel}(X) \to \dots 
\]
as before.
\end{enumerate}
\end{lemma}
\begin{proof}
(ii) Calvo's theorem \ref{nov0201}(ii) implies
that for a smooth affine $R$-scheme $U$ we have weak equivalences 
\[
B_{\dot}\GL(\widehat U \times_{R} \widehat\Delta^{\dot}_R) \xrightarrow{\sim} B_{\dot}\GL(U_{k} \times_{k} \Delta^{\dot}_{k}). 
\]
Since $\holim$ preserves weak equivalences between fibrant simplicial sets, for any open affine covering $U_{\dot} \to X$, viewed as simplicial scheme, we get weak equivalences 
\[
\holim_{[q]\in\Delta} B_{\dot}\GL(\widehat U_{q} \times_{R} \widehat \Delta^{\dot}_R) \xrightarrow{\sim} \holim_{[q]\in\Delta} B_{\dot}\GL((U_{q})_{k}\times_{k} \Delta^{\dot}_{k}).
\]
Taking $\pi_{i}$ and the limit over all finite affine coverings yields isomorphisms
\begin{multline*}
\varinjlim_{U_{\dot}\to X} \pi_{i} \holim_{[q]\in\Delta} B_{\dot}\GL(\widehat U_{q} \times_{R} \widehat \Delta^{\dot}_R) \xrightarrow{\cong} \varinjlim_{U_{\dot}\to X} \pi_{i} \holim_{[q]\in\Delta}  B_{\dot}\GL((U_{q})_{k}\times_{k} \Delta^{\dot}_{k}) \\
\overset{\eqref{eq:AlgKDescent}}{\cong} 
K_{i}(X_{k}).
\end{multline*}
This proves (ii). 
Using Calvo's result $K^{-i}_{\top}(A^{\dag}) \cong K_i(X_k)$
we get (i) for $K^{-*}_{\top}$. Using the five lemma the result for $K_{*}^{\rel}$ follows from this, \eqref{eq:AlgKDescent} again, and (iii).

(iii) follows from Lemma \ref{lem:KRel} and the fact that $\holim$ preserves homotopy fibrations of fibrant simplicial sets \cite[Lemma 5.12]{ThomasonAlgK}.
\end{proof}

\begin{rem}
 Using $K$-dagger or $K$-affinoid algebras in the above constructions we get a notion of topological K-theory for rigid $K$-spaces. It is likely that this coincides with the one defined by Ayoub \cite{Ayoub} using the stable homotopy category of schemes and rigid varieties, respectively.
\end{rem}

\section{Preliminaries on functorial complexes}
\label{sec:Functorial}

For our construction of the relative Chern character we need functorial complexes computing the different cohomology theories involved. The main work has been done before by Huber \cite{Huber}, Besser \cite{Besser}, and Chiarellotto-Ciccioni-Mazzari \cite{Mazzari}. The only difference in our approach is the systematic use of dagger spaces also initiated by Huber and Kings \cite{HK} which simplifies the construction of the rigid and syntomic complexes.

\begin{Nr}[Godement resolutions]\label{Nr:Godement}
We recall some facts on Godement resolutions (see \cite[Sections 3 and 4]{Mazzari} for more details and references). 
To a morphism of sites $u\colon P \to X$ and an abelian sheaf $\mathcal F$ on $X$ one associates a cosimplicial sheaf $[p]\mapsto (u_{*}u^{*})^{p+1}\mathcal F$ on $X$ where the structure maps are induced by the unit and the counit of the adjoint pair $(u^{*}, u_{*})$ between the categories of abelian sheaves on $P$ and $X$. The associated complex of sheaves on $X$ will be denoted by $\Gd_{P}\mathcal F$. There is a canonical augmentation $\mathcal F \to \Gd_{P}\mathcal F$ which is a quasi-isomorphism if $u^{*}$ is exact and conservative (e.g. \cite[Lemma 3.4.1]{Ivorra}). 

We want to use this in the situation where $P$ is a certain set of points of (the topos associated with) $X$ with the discrete topology.
The first case is that of a \emph{scheme} $X$. Here we take $P=Pt(X)$ to be the set of all points of the underlying topological space of $X$. Then $u^{*}$ is given by $\mathcal F \mapsto \coprod_{x\in X}\mathcal F_{x}$ which is exact and conservative and $\mathcal F \xrightarrow{\qis} \Gd_{Pt(X)}\mathcal F$ is the usual Godement resolution.

The second case is that of a dagger space $X$. Here it is not enough to take just the usual points of $X$. Instead, one has to use the set of prime filters on $X$ (introduced in \cite{vdPS}) as alluded to in \cite{Besser} and carried out in \cite[Section 3]{Mazzari}. 
We take $P=Pt(X)$ to be the set of prime filters (\cite[Ex.~3.2.3]{Mazzari}) on the rigid space $X^{\rig}$ associated with $X$ with the discrete topology. Then there are morphisms of sites $Pt(X) \xrightarrow[\text{\cite[3.2.3]{Mazzari}}]{\xi} X^{\rig} \xrightarrow[\text{cf. \ref{Nr:DaggerSpaces}}]{u} X$.

Since $u$ is the identity on underlying G-topological spaces we get from \cite[Lemma 3.2.5]{Mazzari} that $\xi^{*}u^{*}$ is exact and conservative. Hence for any abelian sheaf $\mathcal F$ on $X$ the augmentation $\mathcal F \to \Gd_{Pt(X)}X$ is a quasi-isomorphism. 

It is important to note that in both cases the complex $\Gd_{P}\mathcal F$ consists of flabby sheaves. 
This follows automatically since on a discrete site every sheaf is flabby and direct images of flabby sheaves are flabby.

More generally, if $\mathcal F^{*}$ is a bounded below complex of abelian sheaves on $X$ we can apply $\Gd_{P}$ to each component to get a double complex. We then denote by $\Gd_{P} \mathcal F^{*}$ its associated total complex. It follows from a simple spectral sequence argument that the induced morphism $\mathcal F^{*} \to \Gd_{P}\mathcal F^{*}$ is a quasi-isomorphism.

An important feature of the Godement resolution is its functorial behavior: If 
\[
\xymatrix{ Q\ar[r] \ar[d] & P \ar[d] \\
Y \ar[r]^{f} & X}
\]
is a commutative diagram of sites, and $\mathcal F$ (resp. $\mathcal G$) is a sheaf on $X$ (resp. $Y$), then a morphism $\mathcal F \to f_{*}\mathcal G$ induces a morphism $\Gd_{P}\mathcal F \to f_{*}\Gd_{Q}\mathcal G$ compatible with the augmentations \cite[Lemma 3.1.2]{Mazzari}.
\end{Nr}

\begin{Nr}[{Analytic de Rham cohomology}]\label{mar2005}
 Let $X$ be a smooth $K$-dagger space. There is a notion of differential forms on $X$ (cf. \cite[\S 4]{GKRigid}): For an open affinoid $U=\Sp(A)$ the differential $d\colon A=\Gamma(U,\O_{X}) \to \Gamma(U,\Omega^{1}_{X/K})$ is universal for $K$-derivations of $A$ in finite $A$-modules. As usual one constructs the de Rham complex $\Omega^{*}_{X/K}$ and defines 
\[
H^{*}_{\dR}(X/K) := \Hyp^{*}(X, \Omega^{*}_{X/K}).
\]
We define a complex of $K$-vector spaces, functorial in the $K$-dagger space $X$
\[
R\Gamma_{\dR}(X/K):=\Gamma(X, \Gd_{Pt(X)}\Omega^{*}_{X/K}).
\]
Since the complex $\Gd_{Pt(X)}\Omega^{*}_{X/K}$ consists of flabby sheaves which are acyclic for $\Gamma(X,\,.\,)$ we have natural isomorphisms
\[
H^{*}(R\Gamma_{\dR}(X/K)) \cong H^{*}_{\dR}(X/K).
\]
\end{Nr}

\begin{Nr}[{Algebraic de Rham cohomology}]\label{Nr:AlgdeRham}
 Here $X$ is a smooth, separated scheme of finite type over $K$. Its de Rham cohomology is by definition the hypercohomology of the complex of K\"ahler differential forms:
\[
H^{*}_{\dR}(X/K) := \Hyp^{*}(X, \Omega^{*}_{X/K}).
\]
It is equipped with the \emph{Hodge filtration} constructed as follows: By Nagata's compactification theorem and Hironaka's resolution of singularities there exists a \emph{good compactification of $X$}, i.e. an open immersion $X\overset{j}{\hookrightarrow} \ol X$ of $X$ in a smooth, proper $K$-scheme $\ol X$ such that the complement $D = \ol X - X$ is a divisor with normal crossings. 
 On $\ol X$ one has the complex $\Omega^{*}_{\ol X/K}(\log D)$ of differential forms with logarithmic poles along $D$. There are isomorphisms 
 \[
 H^{*}_{\dR}(X/K) \cong \Hyp^{*}(\ol X,\Omega^{*}_{\ol X/K}(\log D)),
 \]
the Hodge--de Rham spectral sequence $E_{1}^{p,q} = H^{q}(\ol X, \Omega^{p}_{\ol X/K}(\log D)) \Rightarrow  H^{*}_{\dR}(X/K)$ degenerates at $E_{1}$, and the induced filtration $\Fil^\dot H^{*}_{\dR}(X/K)$ is independent of the choice of $\ol X$. It is given by 
\[
\Fil^{n} H^{*}_{\dR}(X/K) = \Hyp^{*}(\ol X, \Omega^{\geq n}_{\ol X/K}(\log D)),
\]
$\geq n$ denoting the naive truncation.

If $f\colon X\to Y$ is a morphism of smooth, separated $K$-schemes of finite type, one can construct good compactifications $X\hookrightarrow \ol X, Y \hookrightarrow \ol Y$ such that $f$ extends to a morphism $\ol f \colon \ol X \to \ol Y$. This implies that the Hodge filtration is functorial. Moreover, the induced map $f^*\colon H^*_\dR(Y/K) \to H^*_\dR(X/K)$ is strict, i.e., $f^*(F^iH^*_\dR(Y/K)) = F^iH^*_\dR(X/K) \cap \operatorname{im}(f^*)$. Indeed,  by the Lefschetz principle and GAGA, this follows from the corresponding fact over $\C$, proven by Deligne.

Since the good compactifications of $X$ form a directed set with respect to maps under $X$ and taking the colimit along a directed set is exact, to get functorial complexes computing algebraic de Rham cohomology together with its Hodge filtration we could take the colimit of the $\Gamma(\ol X, \Gd_{Pt(\ol X)}\Omega^{*}_{\ol X/K}(\log D))$ along the system of good compactifications $\ol X$ of $X$.
However, for the comparison with analytic de Rham cohomology it is technically easier to use the following variant (cf.~\cite[Prop. 4.2.3]{Mazzari}): 

Let $X^{\dag}$ be the dagger analytification of $X$ (cf. \ref{Nr:DaggerSpaces}) and $Pt(X^{\dag})$ as in \ref{Nr:Godement}, $Pt(\ol X)$ the usual set of points of the good compactification $\ol X$ of $X$. We can form the disjoint sum $Pt(X^{\dag}) \sqcup Pt(\ol X)$ viewed as site with the discrete topology to get a commutative diagram of sites
\begin{equation}\label{eq:DiagSites}
\begin{split}
\xymatrix{
Pt(X^{\dag})  \ar[d] \ar[rr] && Pt(X^{\dag}) \sqcup Pt(\ol X) \ar[d] \\
X^{\dag} \ar[r]^{\iota}_{\text{cf. \ref{Nr:DaggerSpaces}}} & X \ar[r]^{j} & \ol X.
}
\end{split}
\end{equation}
There are natural morphisms 
$\Omega^{*}_{\ol X/K}(\log D) \to j_{*}\Omega^{*}_{X/K} \to j_{*}\iota_{*}\Omega^{*}_{X^{\dag}/K}$ which together with \eqref{eq:DiagSites} induce a natural map
\begin{equation}\label{eq:GdCompMap}
\Gd_{Pt(X^{\dag}) \sqcup Pt(\ol X)}\Omega^{*}_{\ol X/K}(\log D) \to j_{*}\iota_{*}\Gd_{Pt(X^{\dag})}\Omega^{*}_{X^{\dag}/K}.
\end{equation}
Thus we are led to define
\begin{equation}\label{eq:FilDRCompl}
\Fil^{n}R\Gamma_{\dR}(X/K) := \varinjlim_{X\hookrightarrow \ol X}\Gamma(\ol X, \Gd_{Pt(X^{\dag}) \sqcup Pt(\ol X)}\Omega^{\geq n}_{\ol X/K}(\log D))
\end{equation}•
where the limit runs over the directed set of good compactifications of $X$. 

It follows from the discussion above that there are natural isomorphisms
\[
H^{*}(\Fil^{n}R\Gamma_{\dR}(X/K)) \cong \Fil^{n}H^{*}_{\dR}(X/K),
\]
and \eqref{eq:GdCompMap} induces natural comparison maps
\begin{equation}\label{eq:ComplCompMap}
\Fil^{n}R\Gamma_{\dR}(X/K) \to R\Gamma_{\dR}(X^{\dag}/K).
\end{equation}
\end{Nr}

\section{The relative Chern character}
\label{sec:Construction}

As before $R$ denotes a complete discrete valuation ring with field of fractions $K$ of characteristic $0$ and residue field $k$ of characteristic $p>0$.
Let $\Sm_R$ be the category of smooth, separated $R$-schemes of finite type.
For $X \in \Sm_{R}$ we have its generic fibre $X_{K}$ with dagger analytification $X_{K}^{\dag}$, and its weak completion $\widehat X$ with generic fibre $\widehat X_{K}$, related by the following morphisms of locally G-ringed spaces
\[
\widehat X_{K} \subseteq X_{K}^{\dag} \xrightarrow{\iota} X_{K}.
\]
In particular, we have morphisms of complexes
\begin{equation}\label{eq:CompDeRhamRigid}
\Fil^{n}R\Gamma_{\dR}(X_{K}/K) \xrightarrow{\eqref{eq:ComplCompMap}} R\Gamma_{\dR}(X_{K}^{\dag}/K) \xrightarrow{\text{by functoriality}} R\Gamma_{\dR}(\widehat X_{K}/K).
\end{equation}
\begin{Nr}\label{dez1201}
We denote by $\Ch$ the category of complexes of abelian groups.
For a morphism $A\xrightarrow{f} B$ in $\Ch$ we denote by $\MF(A\to B) := \Cone(A \to B)[-1]$ the mapping fibre. It has the following property: If $C$ is a complex, the morphisms $C \to \MF(A\to B)$ are in one-to-one correspondence with pairs $(g,h)$ where $g\colon C \to A$ is a morphism of complexes and $h\colon C \to B[-1]$ is a homotopy such that $dh+hd=f\circ g$.
\end{Nr}
\begin{dfn} \label{nov0905}
For every integer $n$ we define a functor $R\Gamma_{\rel}(\,.\,,n)\colon \Sm_{R}^{\op} \to \Ch$ by
\[
R\Gamma_{\rel}(X,n) := \MF(\Fil^{n}R\Gamma_{\dR}(X_{K}/K)\xrightarrow{\eqref{eq:CompDeRhamRigid}} R\Gamma_{\dR}(\widehat X_{K}/K))
\]
and \emph{relative cohomology groups} $H^{*}_{\rel}(X,n) := H^{*}(R\Gamma_{\rel}(X,n))$. 
\end{dfn}
\begin{rems}\label{nov0203}
(i) These are closely related to rigid syntomic cohomology, see Lemma \ref{nov0202} below.
The complex $R\Gamma_{\rel}(X,n)$ can also be interpreted as the syntomic $P$-complex $\mathbb R\Gamma_{\mathrm f,1}(X,n)$ of \cite[2.2]{BesserColeman} for the polynomial $P=1$.

\noindent (ii) Since $X/R$ is smooth, the de Rham cohomology of $\widehat X_{K}$ is just the rigid cohomology of the special fibre $X_{k}$ (see Section \ref{sec:rigidcohom} below). Hence the relative cohomology groups sit in exact sequences
\[
\dots \to H^{i}_{\rel}(X,n) \to \Fil^{n}H^{i}_{\dR}(X_{K}/K) \to H^{i}_{\rig}(X_{k}/K) \to \dots
\]

\noindent (iii) If $X/R$ is proper then $\widehat X_{K} = X_{K}^{\dag}$ (cf. \ref{Nr:DaggerSpaces}) and by GAGA \cite[Kor. 4.5]{GK} $R\Gamma_{\dR}(X_{K}^{\dag}/K) \simeq R\Gamma_{\dR}(X_{K}/K)$, where $\simeq$ denotes a quasi-isomorphism. Hence $R\Gamma_{\rel}(X, n) \simeq \MF(\Fil^{n}R\Gamma_{\dR}(X_{K}/K) \to R\Gamma_{\dR}(X_{K}/K))$ in this case, and the degeneration of the Hodge--de Rham spectral sequence yields isomorphisms
\[
 H^i_\rel(X,n) \cong H^{i-1}_\dR(X_K/K)/\Fil^nH^{i-1}_\dR(X_K/K).
\]

\noindent (iv) For an interpretation of the relative cohomology in terms of stable $\mathbf A^{1}$-homotopy theory we refer the reader to \cite{DegliseMazzari}.
\end{rems}

The goal of this section is to construct relative Chern character maps which will be homomorphisms
\[
\ch_{n,i}^{\rel}\colon K_{i}^{\rel}(X) \to H^{2n-i}_{\rel}(X,n).
\]
We first describe an abstract formalism to obtain homomorphisms from the homotopy groups of certain simplicial sets to the cohomology of suitable functorial complexes and then specialize this to the construction of the relative Chern character and, in the next section, of the syntomic regulator. This formalization of the constructions makes it easier to compare them afterwards. 

\begin{Nr}
\label{dez0701}
We view complexes in $\Ch$ either homologically $\dots\to C_i \xrightarrow{d} C_{i-1} \to \dots$ or cohomologically $\dots\to C^{-i} \xrightarrow{d} C^{-i+1}\to \dots$ using the convention $C_i = C^{-i}$. Given $A, B \in \Ch$, we denote by $\ul\Hom(A,B)$ the mapping complex. In degree $i$ it is given by $\prod_{p} \Hom(A^{p},B^{p+i})$ with differential $f\mapsto f\circ d^{A} - (-1)^{i}d^{B}\circ f$.
In particular, cycles in degree $i$ are given by $Z^i\ul\Hom(A,B)=\Hom_{\Ch}(A,B[i])$.
If $C^{\dot,\dot}$ is a double complex, the differential of the total complex is given on $C^{p,q}$ by $d^{\mathrm{horiz}} + (-1)^{p}d^{\mathrm{vert}}$ ($p$ is the horizontal coordinate).
\end{Nr}
\begin{Nr}
\label{nov1205}
We consider the following setup: $\sS$
is a category, $a\colon \Sm_R \to \sS$ a functor, 
$\Gamma_0\colon \Sm_R^{\op} \to \Ch$ and $\Gamma_1\colon \sS^\op \to \Ch$ are functorial complexes, and we have a natural transformation $\Gamma_0 \to \Gamma_1\circ a$.

For example, $\sS$ could be the category of smooth weak formal $R$-schemes $\Smf_{R}$, $a\colon \Sm_R \to \Smf_{R}$ the 
weak completion functor $X \mapsto \widehat X$, $\Gamma_0 = \Fil^nR\Gamma_\dR((.)_K/K)$, $\Gamma_1 = R\Gamma_\dR((.)_K/K)$ and the natural transformation $\Gamma_0 \to \Gamma_1\circ a$  given by \eqref{eq:CompDeRhamRigid}.

We fix a morphism $E \to B$ in $\Sm_R$. In applications, this will typically be the morphism of simplicial scheme $E_\dot\GL_{r,R} \to B_\dot\GL_{r,R}$. Consider a map $X \xrightarrow{f} B$ in $\Sm_R$ together with a map $a(X) \xrightarrow{g} a(E)$ in $\sS$ such that
\[
 \xymatrix@C+0.5cm@R-0.4cm{
& a(E)\ar[d]\\
a(X) \ar[ur]^g \ar[r]_{a(f)} & a(B)
}
\]
commutes, in other words, an element $(f,g) \in B(X) \times_{a(B)(a(X))} a(E)(a(X))$, where $B(X):= \Hom_{\Sm_R}(X,B)$, etc. By abuse of notation we write $(B\times_{a(B)} a(E))(X)$ for this set.
Then the pair $(f,g)$ gives a commutative diagram
\[
 \xymatrix@C+0.3cm@R-0.4cm{
\Gamma_0(B) \ar[r] \ar[d]_{f^*} & \Gamma_1(a(B)) \ar[d]_{a(f)^*} \ar[r] & \Gamma_1(a(E)) \ar[dl]^-{g^*}\\
\Gamma_0(X) \ar[r] & \Gamma_1(a(X))
}
\]
in $\Ch$ and hence a morphism of complexes, i.e. a zero cycle in the $\ul\Hom$-complex,
\[
 \MF\big(\Gamma_0(B) \to \Gamma_1(a(E))\big) \to \MF \big(\Gamma_0(X) \to \Gamma_1(a(X))\big).
\]
This construction induces a morphism of complexes
\begin{multline*}
 \Z\left[(B\times_{a(B)} a(E))(X)\right] \to  \\
\ul\Hom\left(\MF \big(\Gamma_0(B) \to \Gamma_1(a(E))\big),  \MF \big(\Gamma_0(X) \to \Gamma_1(a(X))\big) \right)
\end{multline*}
where $\Z[\,.\,]$ is the free abelian group considered as a complex in degree $0$.

If $E_\dot \to B_\dot$ is a morphism of simplicial objects in $\Sm_R$, and $X^\dot$ is a cosimplicial object in $\Sm_R$, then $([p],[q])\mapsto\left(B_p\times_{a(B_{p})} a(E_p)\right)(X^q)$ is a bisimplicial set, and we get a natural map of complexes
\begin{multline}\label{dez0702}
 \Tot\Z\left[\left(B_\dot\times_{a(B_{\dot})} a(E_\dot)\right)(X^\dot)\right] \to \\
\ul\Hom\left(\MF\big(\Gamma_0(B_\dot) \to \Gamma_1(a(E_\dot))\big),  \MF \big(\Gamma_0(X^\dot) \to \Gamma_1(a(X^\dot))\big) \right).
\end{multline}
Here $\Gamma_0(B_\dot), \Gamma_0(X^\dot)$, etc. are defined as the direct sum total complexes 
and we view the simplicial, respectively cosimplicial direction as the horizontal one. E.g.~the degree $n$-component $(\Gamma_0(X^\dot))^n$ of the total complex is the possibly infinite direct sum $\bigoplus_{p,q \in \Z, p+q=n} \Gamma_0^q(X^{-p})$.
On the left hand side, the vertical direction is that coming from $X^{\dot}$. Using the sign conventions from \ref{dez0701}, we have to introduce a sign $(-1)^{q(q-1)/2}$ in bidegree $(p,q)$ in order that \eqref{dez0702} is a morphism of complexes.

On homology, \eqref{dez0702} induces for every integer $*$ a map 
\begin{multline}\label{nov0801}
 H_i\left(\Tot\Z\left[\left(B_\dot\times_{a(B_{\dot})} a(E_\dot)\right)(X^\dot)\right]\right) \to \\
\Hom \Big(H^*\left(\MF \big(\Gamma_0(B_\dot) \to \Gamma_1(a(E_\dot))\big)\right), H^{*-i}\left(\MF \big(\Gamma_0(X^\dot) \to \Gamma_1(a(X^\dot))\big) \right) \Big).
\end{multline}
In particular, any class $c\in H^{2n}\left(\MF\left(\Gamma_0(B_\dot) \to \Gamma_1(a(E_\dot))\right)\right)$ gives by composing \eqref{nov0801} with the evaluation at $c$ a map
\begin{multline}\label{nov0901}
\phantom{.}^{*}c\colon H_i\left(\Tot\Z\left[\left(B_\dot\times_{a(B_{\dot})} a(E_\dot)\right)(X^\dot)\right]\right) \to \\
H^{2n-i}\left(\MF \big(\Gamma_0(X^\dot) \to  \Gamma_1(a(X^\dot))\big)\right).
\end{multline}
\end{Nr}%
\begin{dfn}
A \emph{regulator datum} is a tuple $\omega$ consisting of (1) a category $\sS$ together with a functor $a\colon \Sm_{R} \to \sS$, (2) functors $\Gamma_{0}\colon \Sm_{R}^{\op} \to \Ch, \Gamma_{1}\colon \sS^{\op} \to \Ch$ together with a natural transformation $\Gamma_{0} \to \Gamma_{1}\circ a$, (3) a morphism of simplicial objects $E_{\dot}\to B_{\dot}$ in $\Sm_{R}$, and (4) a class $c \in H^{2n}\left(\MF\left(\Gamma_0(B_\dot) \to \Gamma_1(a(E_\dot))\right)\right)$.

To simplify notation, we denote such a regulator datum by $\omega = (\sS, E_{\dot} \to B_{\dot}, \Gamma_{0} \to \Gamma_{1}\circ a, c)$.
\end{dfn}
\begin{lemma}\label{nov1602}
A regulator datum $\omega$ induces for every cosimplicial object $X^{\dot}$ in $\Sm_{R}$ and $i\geq 0$  a homomorphism
\[
\reg_{i}(\omega)\colon \pi_i\left(\left(B_\dot\times_{a(B_{\dot})} a(E_\dot)\right)(X^\dot)\right) \to
H^{2n-i}\left(\MF \big(\Gamma_0(X^\dot) \to  \Gamma_1(a(X^\dot))\big)\right)
\]
\end{lemma}
\begin{proof}
The desired homomorphism $\reg_{i}(\omega)$ is the composition 
\begin{align*}
 & \pi_i\left(\left(B_\dot\times_{a(B_{\dot})} a(E_\dot)\right)(X^\dot)\right)  \\
\longrightarrow& H_i\left(\Z\left[\diag \left(B_\dot\times_{a(B_{\dot})} a(E_\dot)\right)(X^\dot)\right]\right) \qquad \text{ Hurewicz}\\
\overset{\cong}{\longrightarrow}& H_i\left(\Tot \Z\left[\left(B_\dot\times_{a(B_{\dot})} a(E_\dot)\right)(X^\dot)\right]\right) \qquad \text{ by Eilenberg-Zilber}\\
\overset{\!\phantom{.}^{*}c}{\longrightarrow}& H^{2n-i}\left(\MF\left(\Gamma_0(X^\dot) \to  \Gamma_1(a(X^\dot))\right)\right) \qquad \text{ by \eqref{nov0901}} \qedhere
\end{align*}
\end{proof}

We record the following naturality properties which are easily established. They will be used in the comparison of the relative with the syntomic Chern character in the next section.
Consider two regulator data $\omega = (\sS, E_{\dot}\to B_{\dot}, \Gamma_{0}\to\Gamma_{1}\circ a, c)$ and $\omega' = (\sS', E_{\dot}\to B_{\dot}, \Gamma_{0}'\to\Gamma_{1}'\circ a', c')$ with the same $E_{\dot}\to B_{\dot}$.

Assume moreover, that we have a functor $b\colon \sS \to \sS'$ such that $b\circ a \cong a'$, natural transformations $\Gamma_0 \to \Gamma_0'$, $\Gamma_1 \to \Gamma_1'\circ b$ and a natural homotopy $h$ between the compositions $\Gamma_0 \to \Gamma_0' \to \Gamma_1'\circ a'$ and $\Gamma_0 \to \Gamma_1\circ a \to \Gamma_1'\circ b\circ a\cong \Gamma_{1}'\circ a'$.
For every map $Z\to Y$ in $\Sm_R$ these induce a map $\MF\bigl(\Gamma_0(Y) \to \Gamma_1(a(Z))\bigr) \to \MF\big(\Gamma_0'(Y) \to \Gamma_1'(a'(Z))\big)$ (cf. \ref{dez1201}).
\begin{lemma}\label{nov1301}
If $c$ maps to $c'$ by the natural map $H^{2n}\left(\MF\left(\Gamma_0(B_\dot) \to \Gamma_1(a(E_\dot))\right)\right)\to H^{2n}\left(\MF\left(\Gamma_0'(B_\dot) \to \Gamma_1'(a'(E_\dot))\right)\right)$ then for every $X^{\dot}$ the diagram
\[
 \xymatrix@C+0.4cm@R-0.4cm{
\pi_i\left(\left(B_\dot\times_{a(B_{\dot})} a(E_\dot)\right)(X^\dot)\right) \ar[r]^-{\reg_{i}(\omega)} \ar[d] & H^{2n-i}\left(\MF\left(\Gamma_0(X^\dot) \to  \Gamma_1(a(X^\dot))\right)\right) \ar[d] \\
\pi_i\left(\left(B_\dot\times_{a'(B_{\dot})} a'(E_\dot)\right)(X^\dot)\right) \ar[r]^-{\reg_{i}(\omega')} & H^{2n-i}\left(\MF\left(\Gamma_0'(X^\dot) \to  \Gamma_1'(a'(X^\dot))\right)\right)
}
\]
commutes.
\end{lemma}

We now consider regulator data $\omega = (\sS, E_{\dot}\to B_{\dot}, \Gamma_{0}\to\Gamma_{1}\circ a, c)$ and $\omega' = (\sS, E'_{\dot}\to B'_{\dot}, \Gamma_{0}\to\Gamma_{1}\circ a, c')$ with the same category $\sS$ and complexes $\Gamma_{0}, \Gamma_{1}$, and assume that we have a commutative diagram of simplicial objects in $\Sm_R$
\[
 \xymatrix@R-0.3cm@C-0.3cm{
E_\dot \ar[r] \ar[d] & E'_\dot\ar[d]\\
B_\dot \ar[r] & B'_\dot.
}
\]
\begin{lemma}\label{nov1302}
If $c'\in H^{2n}\left(\MF\left(\Gamma_0(B'_\dot) \to \Gamma_1(a(E'_\dot))\right)\right)$ maps to the class $c \in H^{2n}\left(\MF\left(\Gamma_0(B_\dot) \to \Gamma_1(a(E_\dot))\right)\right)$ by the induced map, then for every $X^{\dot}$ the diagram
\[
 \xymatrix{
\pi_i\left(\left(B_\dot\times_{a(B_{\dot})} a(E_\dot)\right)(X^\dot)\right) \ar[r] \ar[dr]_{\reg_{i}(\omega)}&
\pi_i\left(\left(B'_\dot\times_{a(B'_{\dot})} a(E'_\dot)\right)(X^\dot)\right) \ar[d]^{\reg_{i}(\omega')} \\
&  H^{2n-i}\left(\MF\left(\Gamma_0(X^\dot) \to  \Gamma_1(a(X^\dot))\right)\right)
}
\]
commutes.

\end{lemma}

We now construct a regulator datum $\omega^{\rel}_{n,(r)}$ that produces the relative Chern character.
Let $\Smf_{R}$ be the category of smooth weak formal $R$-schemes, $a\colon \Sm_R \to \Smf_{R}$ the 
weak completion functor $X \mapsto \widehat X$, $\Gamma_0$  given by $X \mapsto \Fil^nR\Gamma_\dR(X_K/K)$, $\Gamma_1$ given by $\mathscr X \mapsto R\Gamma_\dR(\mathscr X_K/K)$ and the natural transformation $\Gamma_0 \to \Gamma_1\circ a$  given by \eqref{eq:CompDeRhamRigid}. For $E_\dot \to B_\dot$ we take $E_\dot\GL_{r} \to B_\dot\GL_{r}$.
As cosimplicial object in $\Sm_R$ we will always take $X^\dot=X\times_R\Delta_{R}^\dot = \Spec(A[\Delta^\dot])$ for  some affine $X= \Spec(A)$ in $\Sm_R$.

With these choices we have
\begin{align}
 & \left(B_\dot\times_{a(B_{\dot})} a(E_\dot)\right)(X^\dot) = F_{r}(A), && \text{ cf. \eqref{dez1202},} \label{nov1201}\\
& \MF\left(\Gamma_0(X^\dot) \to  \Gamma_1(a(X^\dot))\right) = R\Gamma_\rel(X\times_{R}\Delta_{R}^\dot, n) && \text{ by Definition \ref{nov0905}.}\label{nov1202}
\end{align}
For the equality \eqref{nov1201} we use that $\widehat\GL_r(\widehat X) = \Hom_{\Smf_{R}}(\widehat X, \widehat\GL_r) \cong \GL_r(A^\dag)$.
To get a regulator datum we need to specify the class $c$. This is accomplished by 
\begin{lemma}\label{nov1204}
For each $r\geq 1$ the natural map 
\begin{multline*}
H^{2n}\left(\MF(\Fil^{n}R\Gamma_{\dR}(B_{\dot}\GL_{r,K}/K) \to R\Gamma_{\dR}(E_{\dot}\widehat\GL_{r,K}/K))\right) \to \\
H^{2n}(\Fil^{n}R\Gamma_{\dR}(B_{\dot}\GL_{r,K}/K)) = \Fil^{n}H^{2n}_{\dR}(B_{\dot}\GL_{r,K}/K)
\end{multline*}
is an isomorphism. 
\end{lemma}
\begin{proof}
This follows from the long exact sequence for the cohomology of a cone together with the fact that $E_{\dot}\widehat\GL_{r,K} $ is a contractible simplicial dagger space, hence has no cohomology in positive degrees (cf. \cite[Lemma 2.11]{TammeDiss}).
\end{proof}
In particular, there is a unique class
\begin{equation}\label{nov1401}
 \ch_{n,(r)}^\rel \in H^{2n}\left(\MF(\Fil^{n}R\Gamma_{\dR}(B_{\dot}\GL_{r,K}/K) \to R\Gamma_{\dR}(E_{\dot}\widehat\GL_{r,K}/K))\right)
\end{equation}
which is mapped to the degree $2n$ component $\ch_{n,(r)}^\dR \in \Fil^{n}H^{2n}_{\dR}(B_{\dot}\GL_{r,K}/K)$ of the universal Chern character class in de Rham cohomology. Since the $\ch_{n,(r)}^{\dR}$ are compatible for varying $r\geq 1$, so are the $\ch_{n,(r)}^{\syn}$.

We also need the following lemma. At this point it is crucial to work with dagger spaces.
\begin{lemma}\label{nov1206}
The natural map $R\Gamma_{\rel}(X,n) \to R\Gamma_{\rel}(X \times \widehat\Delta^{\dot}_{R},n)$ is a quasi-isomorphism for any $X\in \Sm_{R}$.
\end{lemma}
\begin{proof}
It suffices to check this for both components of the cone separately. We first show that
$R\Gamma_{\dR}(\widehat X_{K}/K) \to R\Gamma_{\dR}(\widehat X_{K} \times \widehat\Delta_{K}^{\dot}/K)$
is a quasi-isomorphism. 

By construction $R\Gamma_{\dR}(\widehat X_{K} \times \widehat\Delta_{K}^{\dot}/K)$ is the direct sum total complex of a double complex in the second quadrant. The filtration by columns
gives a convergent spectral sequence in the second quadrant (cf. \cite[5.6.1]{Weibel} for the dual homological case)
\[
E_{1}^{p,q} = H^{q}_{\dR}(\widehat X_{K}\times\widehat\Delta_{K}^{-p}/K) \Rightarrow H^{p+q}(R\Gamma_{\dR}(\widehat X_{K} \times \widehat\Delta_{K}^{\dot}/K)).
\]
The differential $d_{1}$ is induced from the cosimplicial structure of $\widehat\Delta_{K}^{\dot}$. The homotopy invariance of de Rham cohomology of dagger spaces \cite[Prop. 5.8]{GK} implies that $d_{1}^{p,q}$ is the identity if $p$ is even and zero if $p$ is odd. Hence $E_{2}^{0,q} = H^{q}_{\dR}(\widehat X_{K}/K)$, $E_{2}^{p,q}=0$ if $p<0$. It follows that the edge morphism $E_{2}^{0,q} = H^{q}_{\dR}(\widehat X_{K}/K) \to H^{q}(R\Gamma_{\dR}(\widehat X_{K} \times \widehat\Delta_{K}^{\dot}/K))$ is an isomorphism.

For $\Fil^nR\Gamma_\dR(X_K/K) \to \Fil^nR\Gamma_\dR(X_K\times \Delta^\dot_K/K)$ we argue similarly,
using in addition that $H^*_\dR(X_K/K) \to H^*_\dR(X_K\times_K \Delta^{-p}_K/K)$ is an isomorphism by the homotopy invariance of de Rham-cohomology, hence so is $F^nH^*_\dR(X_K/K) \xrightarrow{\cong} F^nH^*_\dR(X_K\times_K \Delta^{-p}_K/K)$ by strictness (cf.~\ref{Nr:AlgdeRham}).
\end{proof}
\begin{dfn}\label{nov1601}
Let $X = \Spec(A)$ be a smooth affine $R$-scheme of finite type. 
Let $\omega^{\rel}_{n,(r)}$ be the regulator datum $(\Smf_{R}, E_{\dot}\GL_{r} \to B_{\dot}\GL_{r}, \Fil^{n}R\Gamma_{\dR}((.)_{K}/K)\to R\Gamma_{\dR}((\widehat.)_{K}/K), \ch_{n,(r)}^{\rel})$.
By \ref{nov1602} this gives  homomorphisms
\[
\reg_{i}(\omega^{\rel}_{n,(r)})\colon \pi_i(F_{r}(A)) \to H^{2n-i}\left(R\Gamma_\rel(X\times \Delta^\dot,n)\right) \overset{\text{\ref{nov1204}}}{\cong}   H^{2n-i}_\rel(X,n),
\]
which are compatible for varying $r\geq 1$.
The \emph{relative Chern character} is defined to be the colimit
\[
 \ch_{n,i}^\rel\colon K_i^\rel(X) = \varinjlim_{r}\pi_i(F_{r}(A)) \xrightarrow{\varinjlim_{r}\reg_{i}(\omega^{\rel}_{n,(r)})} H^{2n-i}_\rel(X,n).
\]

\end{dfn}

We use Jouanolou's trick to extend this definition to all schemes in $\Sm_{R}$. According to Jouanolou and Thomason \cite[4.4]{WeibelKH} such a scheme $X$ admits a Jouanolou torsor $W\xrightarrow{p} X$, i.e. $W$ is affine and $p$ is a torsor for some vector bundle on $X$.
\begin{lemma}
In the above situation the map $p^*\colon R\Gamma_\rel(X,n) \to R\Gamma_\rel(W,n)$ is a quasi-isomorphism.
\end{lemma}
\begin{proof}
It is enough to show that $p$ induces a quasi-isomorphism on both components of the cone. We show that $R\Gamma_\dR(\widehat X_K/K) \xrightarrow{p^*} R\Gamma_\dR(\widehat W_K/K)$ is a quasi-isomorphism. Using moreover the degeneration of the Hodge--de Rham spectral sequence, the proof for $\Fil^nR\Gamma_\dR$ is similar.

Choose a finite open covering $X=\bigcup_{\alpha\in A} U_{\alpha}$ such that $p^{-1}(U_{\alpha}) \to U_{\alpha}$ is isomorphic to a trivial vector bundle $\mathbb A^{r}_{U_{\alpha}} \to U_{\alpha}$.
Let $U_{\dot} \to X$ be the \v{C}ech nerve of this covering and denote by $p^{-1}(U_{\dot}) \to W$ its base change to $W$. 
Since $\{(\widehat U_\alpha)_K\}_{\alpha\in A}$ is an admissible covering of $\widehat X_K$ it follows that $R\Gamma_\dR(\widehat X_K/K) \to R\Gamma_\dR((\widehat U_\dot)_K/K)$ and similarly $R\Gamma_\dR(\widehat W_K/K) \to R\Gamma_\dR(\widehat{p^{-1}(U_\dot)}_K/K)$ are quasi-isomorphisms. 
Hence we are reduced to the case that $W\to X$ is of the form $\mathbb A^r_X \to X$. Then the claim follows from homotopy invariance for the de Rham cohomology of dagger spaces \cite[Prop. 5.8]{GK}
\end{proof}
\begin{dfn}\label{nov0601}
Let $X$ be in $\Sm_{R}$ and choose a Jouanolou torsor $W \xrightarrow{p} X$ . We define the \emph{relative Chern character}
to be the composition
\[
 K_i^\rel(X) \xrightarrow{p^*} K_i^\rel(W) \xrightarrow{\ch_{n,i}^\rel} H^{2n-i}_{\rel}(W, n) \xrightarrow[\cong]{(p^*)^{-1}} H^{2n-i}_{\rel}(X, n).
\]
\end{dfn}
One checks that this does not depend on the choice of $W \to X$ using the fact that for two Jouanolou torsors $W\to X, W'\to X$, the fibre product $W\times_X W' \to X$ is again a Jouanolou torsor.

\section{Comparison with the rigid syntomic regulator}
\label{sec:Comparison}

The main technical problem in the construction of the rigid syntomic regulator is the construction of functorial complexes computing rigid and rigid syntomic cohomology. This was solved by Besser \cite{Besser}. An alternative construction of the regulator using cycle classes and higher Chow groups instead of K-theory is given in \cite{Mazzari}. We recall Besser's construction with some improvements from \cite{Mazzari}. As in \cite{HK} the systematic use of dagger spaces simplifies the construction a little bit.

Let $R$ be as before and assume moreover that the residue field $k$ of $R$ is perfect. Let $K_{0} \subseteq K$ be the field of fractions of the ring of Witt vectors of $k$.

\subsection{Rigid cohomology}
\label{sec:rigidcohom}

We consider the category $\Sch_{k}$ of separated schemes of finite type over $k$ which admit a closed immersion in a flat weakly formal $R$-scheme $\mathscr Y$ with smooth special fibre $\mathscr Y_k$. 
For $X\in\Sm_{R}$ the special fibre $X_{k}$ is in $\Sch_{k}$, as we can take the closed immersion of $X_k$ in the weak completion $\widehat X$ of $X$. 

Let $X$ be in $\Sch_{k}$ and choose an embedding $X \hookrightarrow \mathscr Y$ as above. The \emph{rigid cohomology} of $X$ with coefficients in $K$ is by definition the de Rham cohomology of the tube $]X[_{\mathscr Y} := \spec^{-1}(X) \subset \mathscr Y_{K}$ (cf. \ref{Nr:DaggerSpaces}) of $X$ in $\mathscr Y$:
\[
H^{*}_{\rig}(X/K) = H^{*}_{\dR}(]X[_{\mathscr Y}/K)
\]
\cite{Berthelot}, \cite[Prop. 8.1]{GK}. Up to isomorphism this is independent of the choice of $\mathscr Y$.

Following Besser we define
\[
R\Gamma_{\rig}(X/K)_{\mathscr Y} := R\Gamma_{\dR}(]X[_{\mathscr Y}/K).
\]
This complex is functorial only in the \emph{pair} $(X,\mathscr Y)$. To obtain complexes functorial in the $k$-scheme $X$ we proceed as in \cite[\S 4]{Besser}.
Define the category of \emph{rigid pairs} $\mathcal{RP}$: Objects are pairs $(X, j\colon X \hookrightarrow \mathscr Y)$ where $X$ and $j$ are as above. We will often abbreviate such a pair as $(X, \mathscr Y)$. Morphisms $(X', \mathscr Y') \to (X, \mathscr Y)$ are pairs  of morphisms $(f\colon X'\to X, F\colon ]X'[_{\mathscr Y'} \to ]X[_{\mathscr Y})$ such that the diagram
\[\xymatrix{
]X'[_{\mathscr Y'} \ar[d]_{\spec} \ar[r]^F & ]X[_{\mathscr Y} \ar[d]^{\spec}\\
X' \ar[r]^f & X
}
\]
commutes.
The category $\mathcal{RP}$ substitutes Besser's category $\mathcal{RT}$ of rigid triples. Note that there is a natural functor $\Sm_{R}\to \mathcal{RP}$ taking $X$ to the pair $(X_{k}, \widehat X)$. 

With this replacement Besser's construction goes through word by word and yields the following
\begin{prop}[{\cite[Prop. 4.9, Cor. 4.22]{Besser}}]\label{prop:BesserRigComplexes} 
\begin{enumerate}
\item There exists a functor $R\Gamma_{\rig}(\,.\,,/k)\colon \Sch_{k}^{\op} \to \Ch_{K}$ to the category of complexes of $K$-vector spaces
such that $H^{*}(R\Gamma_{\rig}(X/K)) = H^{*}_{\rig}(X/K)$ functorially. If $K$ is absolutely unramified, i.e. $K = K_{0}$, and $\sigma$ is the Frobenius on $K_{0}$, there exists a natural $\sigma$-semilinear Frobenius-endomorphism $\phi$ on $R\Gamma_{\rig}(X/K_{0})$.
\item There exists a functor $\mathcal{RP}^{\op} \to \Ch_{K}$, $(X, \mathscr Y) \mapsto \widetilde{R\Gamma}_{\rig}(X/K)_{\mathscr Y}$ together with $\mathcal{RP}$-functorial quasi-isomorphisms
\[
R\Gamma_{\dR}(]X[_{\mathscr Y}/K) = R\Gamma_{\rig}(X/K)_{\mathscr Y} \xleftarrow{\qis} \widetilde{R\Gamma}_{\rig}(X/K)_{\mathscr Y} \xrightarrow{\qis} R\Gamma_{\rig}(X/K). 
\]
\end{enumerate}
\end{prop}

\subsection{Rigid syntomic cohomology}

\begin{Nr}\label{Nr:QuasiPullback}
Recall that the \emph{homotopy pullback} of a diagram of complexes $A \xrightarrow{f} C \xleftarrow{g} B$ is by definition the complex $A\tilde\times_{C}B := \Cone(A\oplus B \xrightarrow{f-g} C)[-1]$. It fits in a diagram 
\[
\xymatrix{
A \tilde\times_{C} B \ar[r]^-{\tilde f} \ar[d]_{\tilde g} & B \ar[d]^{g}\\
A \ar[r]^{f} & C
}
\]
which is commutative up to \emph{canonical} homotopy, given by the projection to the $C$-component of the cone. If $f$ is a quasi-isomorphism, so is $\tilde f$.
\end{Nr}

\begin{Nr}\label{Nr:TildeComplexes}
Let $X$ be in $\Sm_{R}$. Then $]X_{k}[_{\widehat X} = \widehat X_{K} \subseteq X_{K}^{\dag}$ and we have natural maps of complexes, functorial in $X$,
\[
\Fil^{n}R\Gamma_{\dR}(X_{K}/K) \xrightarrow{\eqref{eq:CompDeRhamRigid}} R\Gamma_{\dR}(]X_{k}[_{\widehat X}/K) \xleftarrow{\qis} \widetilde{R\Gamma}_{\rig}(X_{k}/K)_{\widehat X} \xrightarrow{\qis} R\Gamma_{\rig}(X_{k}/K).
\] 
Define $\Fil^{n}\widetilde{R\Gamma}_{\dR}(X/K)$ to be the homotopy pullback of the left two arrows above, a complex quasi-isomorphic to $\Fil^{n}R\Gamma_{\dR}(X_{K}/K)$ which admits a natural map
\begin{equation}\label{jul211}
\Fil^{n}\widetilde{R\Gamma}_{\dR}(X/K) \to \widetilde{R\Gamma}_{\rig}(X_{k}/K)_{\widehat X}. 
\end{equation}
On $R\Gamma_{\rig}(X_{k}/K_{0})$ we have the Frobenius $\phi$ and the natural map to $R\Gamma_{\rig}(X_{k}/K)$. We define the complex
\[
\Phi(n)(X_{k}) :=\Cone\left(R\Gamma_{\rig}(X_{k}/K_{0}) \to R\Gamma_{\rig}(X_{k}/K_{0}) \oplus R\Gamma_{\rig}(X_{k}/K)  \right)
\]
where the map is given by $\omega\mapsto \left((1-\frac{\phi}{p^{n}})\omega,\omega\right)$.
This complex is functorial in the $k$-scheme $X_{k}$ and there are natural maps 
\begin{equation}\label{jul212}
\widetilde{R\Gamma}_{\rig}(X_{k}/K)_{\widehat X} \xrightarrow{\simeq} R\Gamma_{\rig}(X_{k}/K)  \to \Phi(n)(X_{k}).
\end{equation}
Consider the composition
\begin{equation}\label{eq:MapFilPhi}
\Fil^{n}\widetilde{R\Gamma}_{\dR}(X/K) \xrightarrow{\eqref{jul211}} \widetilde{R\Gamma}_{\rig}(X_{k}/K)_{\widehat X} \xrightarrow{\eqref{jul212}} \Phi(n)(X_{k}).
\end{equation}
%
\end{Nr}
\begin{dfn}
We define the \emph{syntomic complex of $X$ twisted by $n$}
\[
R\Gamma_{\syn}(X,n):=\MF(\Fil^{n}\widetilde{R\Gamma}_{\dR}(X/K) \xrightarrow{-\eqref{eq:MapFilPhi}} \Phi(n)(X_{k})).
\]
Its cohomology groups will be denoted by $H^{*}_{\syn}(X,n)$. 
\end{dfn}
\begin{rem}
Writing down the iterated cone construction explicitly one sees that $R\Gamma_{\syn}(X,n)$ is isomorphic to the complex
\[
\Cone\left(R\Gamma_{\rig}(X_{k}/K_{0}) \oplus \Fil^{n}\widetilde{R\Gamma}_{\dR}(X/K) \to R\Gamma_{\rig}(X_{k}/K_{0}) \oplus R\Gamma_{\rig}(X_{k}/K)\right)[-1]
\]
where the map is given by $(x,y) \mapsto ((1-\frac{\phi}{p^{n}})x, x-y)$. From this one sees that the fundamental Proposition 6.3 of \cite{Besser} also holds for our definition of rigid syntomic cohomology. Hence all further constructions of \cite{Besser} work equally well in our setting. In particular, there are natural maps
\begin{equation}\label{eq:MapSynDeRham}
H^{*}_{\syn}(X,n) \to \Fil^{n}H^{*}_{\dR}(X_{K}/K).
\end{equation}

In fact, it is possible to construct a natural chain of quasi-isomorphisms connecting our version of the rigid syntomic complex with Besser's. But since we do not need this we omit the lengthy and technical details.
\end{rem}

We now give Besser's construction of the rigid syntomic regulator \cite[Thm. 7.5]{Besser} in the setup of the present paper. By \cite[Prop. 7.4]{Besser} and the discussion following it, for every positive integers $n, r$ there exists a class 
\begin{equation}\label{nov1501}
\ch_{n,(r)}^{\syn} \in H^{2n}_{\syn}(B_{\dot}\GL_{r,R}, n),
\end{equation} 
the universal $n$-th syntomic Chern character class, uniquely determined by the fact that it is mapped to the degree $2n$ component $\ch_{n,(r)}^\dR$ of the universal  de Rham Chern character class in $\Fil^{n}H^{2n}_{\dR}(B_{\dot}\GL_{r,K}/K)$ under \eqref{eq:MapSynDeRham}. These are compatible for varying $r$.

Let 
$a=(\,.\,)_{k}\colon \Sm_{R} \to \Sch_{k}$ be the special fibre functor.
We define the syntomic regulator datum $\omega^{\syn}_{n,(r)} := (\Sch_{k}, B_{\dot}\GL_{r,R} \xrightarrow{\id} B_{\dot}\GL_{r,R}, \Fil^{n}\widetilde{R\Gamma}_{\dR}(\,.\,/K) \xrightarrow{-\eqref{eq:MapFilPhi}} \Phi(n)((\,.\,)_{k}), \ch_{n,(r)}^{\syn})$.

Now let $X=\Spec(A)$ be an affine scheme in $\Sm_{R}$ and as always $X^{\dot} = X \times_{R} \Delta_{R}^{\dot}$.
With these choices we have
\begin{align*}
 & \left(B_\dot\GL_{r,R}\times_{a(B_{\dot}\GL_{r,R})} a(B_\dot\GL_{r,R})\right)(X^\dot) = B_\dot\GL_{r,R}(A[\Delta^\dot]), \\
& \MF\left(\Fil^{n}\widetilde{R\Gamma}_{\dR}(X^\dot) \to  \Phi(n)(X_{k}^\dot)\right) = R\Gamma_\syn(X\times_{R}\Delta_{R}^\dot, n), \\ 
& \MF\left(\Fil^{n}\widetilde{R\Gamma}_{\dR}(B_\dot\GL_{r,R}) \to \Phi(n)(B_\dot\GL_{r,k})\right) = R\Gamma_\syn(B_\dot\GL_{r,R},n).
\end{align*}
Similarly as in Lemma \ref{nov1206} one shows that $R\Gamma_\syn(X,n) \to R\Gamma_\syn(X\times_{R}\Delta_{R}^\dot, n)$ is a quasi-isomorphism.
\begin{dfn}
For $i\geq 1$ the \emph{syntomic Chern character} or \emph{regulator} is given by 
 \begin{multline*}
 \ch_{n,i}^{\syn}\colon K_i(X) = \varinjlim_{r} \pi_{i}\left(B_{\dot}GL_{r}(A[\Delta^{\dot}])\right) \xrightarrow{\varinjlim_{r}\reg_{i}(\omega^{\syn}_{n,(r)})} \\
 H^{2n-i}(R\Gamma_\syn(X\times_{R}\Delta_{R}^\dot, n)) \cong  H^{2n-i}_\syn(X,n).
 \end{multline*}
\end{dfn}
Using that the natural map $B_\dot\GL(A) \to \diag B_\dot\GL(A[\Delta^\dot])$ induces an isomorphism in homology with $\Z$-coefficients (cf.~the proof of Lemma \ref{mar2003}), it is easy to check that this construction is equivalent to Besser's in the affine case.

Again, this is extended to all schemes in $\Sm_{R}$ using Jouanolou's trick (cf.~\ref{nov0601}).

\subsection{The comparison}

\begin{lemma}\label{nov0202}
There exist complexes $\widetilde{R\Gamma}_\rel(X,n)$, functorial in $X\in \Sm_{R}$, together with maps
\[
 R\Gamma_\rel(X,n) \xleftarrow{\simeq} \widetilde{R\Gamma}_\rel(X,n) \to R\Gamma_\syn(X,n),
\]
the left pointing arrow being a quasi-isomorphism. These induce natural maps
\begin{equation}\label{mar0701}
H^{*}_{\rel}(X,n) \to H^{*}_{\syn}(X, n)
\end{equation}
which are isomorphisms if $X$ is proper and $*\not\in \{2n, 2n+1, 2n+2\}$.
\end{lemma}
\begin{proof}
Consider the following diagram of complexes
\begin{equation}\label{eq:MapRelSyn}
\begin{split}
\xymatrix@C+0.3cm@R-0.3cm{
\Fil^{n}R\Gamma_{\dR}(X_{K}/K)  \ar[d]^{\eqref{eq:CompDeRhamRigid}} & \Fil^{n}\widetilde{R\Gamma}_{\dR}(X/K) \ar[d]^{\eqref{jul211}} \ar[l]_-\simeq \ar@{=}[r] & \Fil^{n}\widetilde{R\Gamma}_{\dR}(X/K) \ar[d]^{-\eqref{eq:MapFilPhi}}\\
R\Gamma_{\dR}(\widehat X_{K}/K) & \widetilde{R\Gamma}_{\rig}(X_{k}/K)_{\widehat X} \ar[l]_-{\simeq} \ar[r]^{-\eqref{jul212}} & \Phi(n)(X_{k})
}
\end{split}
\end{equation}
The left square commutes up to canonical homotopy (cf. \ref{Nr:QuasiPullback}), the right square strictly commutes. 
We set $\widetilde{R\Gamma}_\rel(X,n) = \MF(\Fil^{n}\widetilde{R\Gamma}_{\dR}(X/K) \to \widetilde{R\Gamma}_{\rig}(X_{k}/K)_{\widehat X})$.
The desired maps are induced by the maps in the diagram together with the homotopy which makes the left hand square commute (cf. \ref{dez1201}).

The second statement follows from weight considerations as in the proof of \cite[Prop. 8.6]{Besser}.
\end{proof}

We can now formulate the main result of this paper:
\begin{thm}\label{thm:Comparison}
For every $X$ in $\Sm_{R}$ and $i \geq 1$ the diagram
\[
\xymatrix@C+0.5cm{
K_{i}^{\rel}(X) \ar[r] \ar[d]_{\ch_{n,i}^{\rel}} & K_{i}(X) \ar[d]^{\ch_{n,i}^{\syn}} \\
H^{2n-i}_{\rel}(X,n) \ar[r]^-{\eqref{mar0701}} & H^{2n-i}_{\syn}(X,n)
}
\]
commutes.
\end{thm}
\begin{proof}
By construction of the maps in the diagram, we may suppose that $X = \Spec(A)$ is affine. Write $A_{k} = A\otimes_{R} k$, so that $X_{k} = \Spec(A_{k})$. 

We split the diagram up into the following smaller diagrams, and show that every single one of them commutes.
\begin{equation}\begin{split}\label{nov1502}
 \xymatrix@C+0.3cm{
K_i^\rel(X)\ar[d]_{\ch_{n,i}^{\rel}} \ar[dr]^{\widetilde\ch_{n,i}^\rel} \ar[r] & K_i(X,X_k) \ar[r] \ar[dr]^{\widetilde{\ch}_{n,i}^\syn} & K_i(X) \ar[d]^{\ch_{n,i}^\syn} \\
H^{2n-i}_\rel(X,n) & H^{2n-i}(\widetilde{R\Gamma}_\rel(X,n)) \ar[l]^-{\text{cf. \ref{nov0202}}}_-{\cong} \ar[r]_-{\text{cf. \ref{nov0202}}}& H^{2n-i}_\syn(X,n)
}
\end{split}
\end{equation}
Here $K_{i}(X, X_{k}) := \pi_{i}\left(B_{\dot}\GL(A[\Delta^{\dot}]) \times_{B_{\dot}\GL(A_{k}[\Delta^{\dot}])} E_{\dot}\GL(A_{k}[\Delta^{\dot}])\right)$. Since $A_{k}[\Delta^{\dot}] \cong A^{\dag}\<\Delta^{\dot}\>^\dag \otimes_{R} k$ the map $K_{i}^{\rel}(X) \to K_{i}(X)$ factors through $K_{i}(X,X_{k})$. 

All vertical resp. diagonal maps are induced by a compatible family of maps for each finite level $r$. We can thus restrict to a fixed finite level $r$.
To ease notation we write $E_{\dot} := E_{\dot}\GL_{r,R}, B_{\dot}:=B_{\dot}\GL_{r,R}$.
The diagonal maps arise as follows: 
Since the left hand square in \eqref{eq:MapRelSyn} commutes up to canonical homotopy we are in the situation of Lemma \ref{nov1301}. In particular, we have a natural isomorphism (cf. \eqref{eq:MapRelSyn})
\begin{multline*}
H^{2n}(\MF(\Fil^{n}\widetilde{R\Gamma}_{\dR}(B_{\dot}/K) \to \widetilde{R\Gamma}_{\rig}(E_{\dot,k}/K)_{\widehat E_{\dot}}))
\xrightarrow{\cong} \\
H^{2n}\left(\MF(\Fil^{n}R\Gamma_{\dR}(B_{\dot}\GL_{r,K}/K) \to R\Gamma_{\dR}(E_{\dot}\widehat\GL_{r,K}/K))\right)
\end{multline*}
and we define $\widetilde\ch^{\rel}_{n,(r)}$ to be the class mapping to $\ch^{\rel}_{n,(r)}$ (cf. \eqref{nov1401}) under this isomorphism.
We let
$\widetilde\ch_{n,i}^{\rel}$ be induced by the regulator data $\widetilde\omega^{\rel}_{n,(r)} = (\Smf_{R}, E_{\dot} \to B_{\dot}, \Fil^n\widetilde{R\Gamma}_\dR(./K) \to \widetilde{R\Gamma}_\rig((.)_{k}/K)_{\widehat{(.)}}, \widetilde\ch^{\rel}_{n,(r)})$.
 Then it is clear from Lemma \ref{nov1301} and the constructions that the left triangle commutes.

The map $\widetilde\ch^\syn_{n,i}$ is induced by the data $\widetilde\omega^\syn_{n,(r)} = (\Sch_k, E_\dot \to B_\dot, \Fil^{n}\widetilde{R\Gamma}_{\dR}(./K) \xrightarrow{-\eqref{eq:MapFilPhi}} \Phi(n)((.)_{k}), \widetilde\ch_{n,(r)}^\syn)$ where $\widetilde\ch_{n,(r)}^{\syn}$ is defined as follows: We have a commutative diagram
\[
\xymatrix@R-0.3cm@C-0.3cm{
E_{\dot} \ar[r] \ar[d] & B_{\dot}\ar[d]^{\id}\\
B_{\dot} \ar[r]^-{\id} & B_{\dot}
}
\]
which induces a map 
\begin{multline*}
H^{2n}_{\syn}(B_{\dot}, n) \cong H^{2n}\left(\MF\left(\Fil^{n}\widetilde{R\Gamma}_{\dR}(B_\dot/K) \to \Phi(n)(B_{\dot,k})\right)\right) \to \\
H^{2n}\left(\MF\left(\Fil^{n}\widetilde{R\Gamma}_{\dR}(B_\dot/K) \to \Phi(n)(E_{\dot,k})\right)\right)
\end{multline*}
and $\widetilde\ch_{n,(r)}^{\syn}$ is by definition the image of $\ch_{n,(r)}^{\syn}$ (see \eqref{nov1501}) by this map. It is then clear from Lemma \ref{nov1302} that the right triangle in \eqref{nov1502} commutes. 

It remains to show that the middle parallelogram in \eqref{nov1502} commutes. For this we apply Lemma \ref{nov1301} to the regulator data $\widetilde\omega^{\rel}_{n}$ and $\widetilde\omega^{\syn}_{n}$: The special fibre functor $\Sm_{R} \to \Sch_{k}$ factors naturally as $\Sm_{R}\xrightarrow{X\mapsto \widehat X} \Smf_{R} \xrightarrow{\mathscr X\mapsto \mathscr X_{k}} \Sch_{k}$. Moreover, we have a natural transformation between functors $\Smf_{R}^{\op}\to \Ch, \widetilde{R\Gamma}_{\rig}((.)_{k}/K)_{(.)} \to \Phi(n)((.)_{k})$ given by $-\eqref{eq:MapFilPhi}$. Since
\[
\xymatrix@C+0.1cm@R-0.5cm{
\Fil^{n}\widetilde{R\Gamma}_{\dR}(X/K) \ar[d]^{\eqref{jul211}}  \ar@{=}[r] & \Fil^{n}\widetilde{R\Gamma}_{\dR}(X/K) \ar[d]^{-\eqref{eq:MapFilPhi}}\\
\widetilde{R\Gamma}_{\rig}(X_{k}/K)_{\widehat X} \ar[r]^{-\eqref{jul212}} & \Phi(n)(X_{k})
}
\]
commutes for every $X\in \Sm_{R}$,
we have a natural map
\begin{multline}\label{nov1503}
H^{2n}\left(\MF(\Fil^{n}\widetilde{R\Gamma}_{\dR}(B_{\dot}/K) \to \widetilde{R\Gamma}_{\rig}(E_{\dot,k}/K)_{\widehat E_{\dot}})\right) \to \\
H^{2n}\left(\MF(\Fil^{n}\widetilde{R\Gamma}_{\dR}(B_\dot/K) \to \Phi(n)(E_{\dot,k}))\right).
\end{multline}
If we show that under this map $\widetilde\ch_{n,(r)}^{\rel}$ maps to $\widetilde\ch_{n,(r)}^{\syn}$ then Lemma \ref{nov1301} implies the desired commutativity.
But indeed, \eqref{nov1503} fits in a commutative diagram
\[
\xymatrix@C+0.2cm@R-0.4cm{
H^{2n}\left(\MF(\Fil^{n}\widetilde{R\Gamma}_{\dR}(B_{\dot}/K) \to \widetilde{R\Gamma}_{\rig}(E_{\dot,k}/K)_{\widehat E_{\dot}})\right) \ar[r]^-{\cong} \ar[d]_{\eqref{nov1503}} & H^{2n}( \Fil^{n}\widetilde{R\Gamma}_{\dR}(B_{\dot}/K))\\
H^{2n}\left(\MF(\Fil^{n}\widetilde{R\Gamma}_{\dR}(B_\dot/K) \to \Phi(n)(E_{\dot,k}))\right) \ar[ur]_-{\cong} & \Fil^{n}H^{2n}_{\dR}(B_{\dot}\GL_{r,K}/K) \ar@{=}[u]
}
\]
where the two isomorphisms are established as in Lemma \ref{nov1204} and, by the constructions, both, $\widetilde\ch_{n,(r)}^{\rel}$ and $\widetilde\ch_{n,(r)}^{\syn}$, map to $\ch_{n,(r)}^{\dR}$ on the right hand side.
\end{proof}

\subsection{Applications}\label{ssec:applications}
In this  section we assume that the residue field $k$ of $R$ is finite.

Let $X$ be a smooth and proper $R$-scheme. The \'etale Chern character class induces a map $\ch_{n,i}^{\et}\colon K_{i}(X) \to H^{2n-i}_{\et}(X_{K}, \Q_{p}(n))$. It follows from the crystalline Weil conjectures \cite{ChiarellottoLeStum} and Faltings' crystalline comparison theorem \cite{FaltingsCrystalline} that $H^{2n-i}_{\et}(X_{\ol K}, \Q_{p}(n))^{G_{K}}=0$ for $i>0$, where $\ol K$ denotes an algebraic closure of $K$ and $G_{K}=\Gal(\ol K/K)$. Hence the Hochschild-Serre spectral sequence for $X_{\ol K} \to X_{K}$ induces an edge morphism $H^{2n-i}_{\et}(X_{K}, \Q_{p}(n)) = \Fil^{1}H^{2n-i}_{\et}(X_{K}, \Q_{p}(n)) \to H^{1}(G_{K}, H^{2n-i-1}_{\et}(X_{\ol K}, \Q_{p}(n)))$ and the composition
\begin{equation}
r_{p}\colon K_{i}(X) \xrightarrow{\ch_{n,i}^{\et}} H^{2n-i}_{\et}(X_{K}, \Q_{p}(n)) \xrightarrow{\text{edge}} H^{1}(G_{K}, H^{2n-i-1}_{\et}(X_{\ol K}, \Q_{p}(n)))
\end{equation}•
is the \emph{\'etale $p$-adic regulator}.

According to the de Rham comparison theorem \cite{FaltingsCrystalline} we have an isomorphism of filtered vector spaces
\[
D_{\dR}(H^{2n-i-1}_{\et}(X_{\ol K}, \Q_{p}(n))) \cong H^{2n-i-1}_{\dR}(X_{K}/K)(n)
\]
where the twist by $n$ on the right hand side only shifts the filtration. Hence
\[
D_{\dR}(H^{2n-i-1}_{\et}(X_{\ol K}, \Q_{p}(n)))/\Fil^{0} \cong H^{2n-i-1}_{\dR}(X_{K}/K)/\Fil^{n} \cong H^{2n-i}_{\rel}(X,n)
\]
(see Remark \ref{nov0203}(iii)).
In particular, the Bloch-Kato exponential for the $G_{K}$-re\-pre\-sen\-ta\-tion $H^{2n-i-1}_{\et}(X_{\ol K}, \Q_{p}(n))$ is a map
\[
\exp\colon H^{2n-i-1}_{\dR}(X_{K}/K)/\Fil^{n} \to H^{1}(G_{K}, H^{2n-i-1}_{\et}(X_{\ol K}, \Q_{p}(n)))
\]
and from Nizio\l's work we get the
\begin{cor}\label{apr1201}
For each smooth, projective $R$-scheme $X$ the diagram
\[
\xymatrix@C+0.5cm{
K_{i}^{\rel}(X) \ar[d]^{\ch_{n,i}^{\rel}} \ar[r] & K_{i}(X) \ar[d]^{r_{p}}\\
H^{2n-i-1}_{\dR}(X_{K}/K)/\Fil^{n} \ar[r]^-{\exp} & H^{1}(G_{K}, H^{2n-i-1}_{\et}(X_{\ol K}, \Q_{p}(n)))
}
\]
commutes.
\end{cor}
\begin{proof}
It follows from Nizio\l's work \cite{NiziolImg, NiziolCrys} and the comparison with Besser's syntomic cohomology 
\cite[Prop. 9.9]{Besser} that there is a natural map $H^{*}_{\syn}(X,n) \to H_{\et}^{\ast}(X_{K}, \Q_{p}(n))$ which is compatible with Chern classes \cite[Cor. 9.10]{Besser}.  By \cite[Prop. 9.11]{Besser} the composition $H^{2n-i-1}_{\dR}(X_{K}/K)/\Fil^{n} \to H^{2n-i}_{\syn}(X,n) \to H^{2n-i}_{\et}(X_{K}, \Q_{p}(n)) \to H^{1}(G_{K}, H^{2n-i-1}_{\et}(X_{\ol K}, \Q_{p}(n)))$ 
is the Bloch-Kato exponential for $H^{2n-i-1}_{\et}(X_{\ol K}, \Q_{p}(n))$. Hence the claim follows from the comparison of the relative Chern character with the syntomic regulator in Theorem \ref{thm:Comparison}.
\end{proof}
\begin{rem}\label{nov1504}
By Calvo's result (Lemma \ref{lemma:RelKSmoothSeparated}(ii)) we have an isomorphism $K_{\top}^{-i}(X) \cong K_{i}(X_{k})$. 
For $X_{k}$ smooth and projective these groups are conjectured to be torsion (Parshin's conjecture). This would imply that $K_{i}^{\rel}(X) \to K_{i}(X)$ is rationally an isomorphism.
It follows from \cite{Harder} and \cite{QuillenK} that this conjecture is true for $\dim X_{k} \leq 1$. 
\end{rem}

From the previous corollary together with our earlier work \cite{TammeBorel, TammeDiss}, we get a new proof of the main result of \cite{HK}: Huber and Kings introduce the \emph{$p$-adic Borel regulator} $r_{\mathrm{Bo},p}\colon K_{2n-1}(R) \to K$ by imitating the construction of the classical Borel regulator for the field of complex numbers, replacing the van Est isomorphism by the Lazard isomorphism.
\begin{cor}[Huber-Kings \cite{HK}] \label{cor:HK}
Let $K$ be a finite extension of $\Q_{p}$. The diagram
\[
\xymatrix{
K_{2n-1}(R) \ar[d]_{\frac{(-1)^{n}}{(n-1)!}r_{\mathrm{Bo},p}} \ar[dr]^{r_{p}}\\
K=D_{\dR}(\Q_{p}(n)) \ar[r]^-{\exp} & H^{1}(G_{K}, \Q_{p}(n))
}
\]
commutes.
\end{cor}
\begin{rem*}
The factor $\frac{(-1)^{n}}{(n-1)!}$ appears since Huber and Kings use Chern classes in the normalization of both, the \'etale and the $p$-adic Borel regulator, whereas we used Chern character classes in the definition of the \'etale regulator.
\end{rem*}

\begin{proof}
Apply the previous corollary with $X=\Spec(R)$, $i=2n-1$. 
As mentioned in Remark \ref{nov1504},
$K_{2n-1}^{\rel}(R) \to K_{2n-1}(R)$ is rationally an isomorphism.
In Theorem \ref{thm:KaroubisRelChern}  we will show that the present version of the relative Chern character coincides with Karoubi's original construction. For this we showed in \cite[Corollary 7.23]{TammeDiss} that the diagram 
\[
\xymatrix@R-0.2cm{
K^{\rel}_{2n-1}(R) \ar[r] \ar[d]_{\ch_{n,2n-1}^{\rel}} & K_{2n-1}(R) \ar[dl]^{\frac{(-1)^{n}}{(n-1)!}r_{\mathrm{Bo},p}} \\
K
}
\]
 commutes.
\end{proof}

\section{Comparison with Karoubi's original construction}
\label{sec:Karoubi}

As before $R$ denotes a complete discrete valuation ring with field of fractions $K$ of characteristic $0$ and residue field $k$ of characteristic $p>0$.
In this section we compare the relative Chern character of Section \ref{sec:Construction} with Karoubi's original construction \cite{KarCR, Kar87} in the form of \cite{TammeDiss}. For any smooth, affine $R$-scheme $X=\Spec(A)$ this is a homomorphism
\[
\ch_{n,i}^{\mathrm{Kar}}\colon K_{i}^{\rel}(X) \to \Hyp^{2n-i-1}(\widehat X_{K}, \Omega^{<n}_{\widehat X_{K}/K}) = H^{2n-i-1}(\Omega^{<n}(\widehat X_{K})).
\]
We recall the main steps in its construction.

\begin{Nr}
We begin with some preliminaries concerning integration (see also \cite[Appendix]{TammeBorel}). Consider the polynomial ring $\Q[\Delta^{p}] = \Q[x_{0},\dots, x_{n}]/(\sum_{i}x_{i}-1)$. There is a well defined integration map $\int_{\Delta^{p}}\colon \Omega^{p}_{\Q[\Delta^{p}]/\Q} \to \Q$ sending an algebraic $p$-form $\omega$ to the integral of $\omega$, considered as a smooth $p$-form, over the real standard simplex $\mathbf \Delta^{p}:=\{(x_{0}, \dots, x_{p}) \in \R^{p+1}\,|\,\sum_{i} x_{i} = 1, \forall i: 0\leq x_{i}\leq 1\} \subseteq \R^{p+1}$ with orientation given by the form $dx_{1}\cdots dx_{p}$. This integral is in fact a rational number. It satisfies Stokes's formula
\begin{equation}\label{mar2006}
\int_{\Delta^{p}} d\omega = \sum_{i=0}^{p} (-1)^{i} \int_{\Delta^{p-1}} (\del^{i})^{*}\omega.
\end{equation}

Similarly one can define an integration map $\int_{\Delta^{p}\times\Delta^{q}}\colon \Omega^{p+q}_{\Q[\Delta^{p}]\otimes_{\Q}\Q[\Delta^{q}]/\Q} \to \Q$, again given by the usual integration over the product of real simplices $\mathbf \Delta^{p}\times \mathbf \Delta^{q}$. There is a canonical decomposition of $\mathbf \Delta^{p}\times \mathbf \Delta^{q}$ into copies of the standard $p+q$-simplex $\mathbf \Delta^{p+q}$, indexed by all $(p,q)$-shuffles $(\mu,\nu)$ (cf.~the proof of Theorem \ref{thm:KaroubisRelChern} below). It follows from this and the analogous formula for smooth differential forms that 
\begin{equation}\label{mar2007}
\int_{\Delta^{p}\times\Delta^{q}} \omega = \sum_{(\mu,\nu)} \sgn(\mu,\nu) \int_{\Delta^{p+q}} (\mu,\nu)^{*}\omega,
\end{equation}
where the sum runs over all $(p,q)$-shuffles and $(\mu,\nu)^{*}\omega\in \Omega^{p+q}_{\Q[\Delta^{p+q}]/\Q}$ is the pullback of $\omega \in \Omega^{p+q}_{\Q[\Delta^{p}]\otimes_{\Q}\Q[\Delta^{q}]/\Q}$ to the simplex corresponding to the shuffle $(\mu,\nu)$. 

Tensoring the $\Q$-linear integration map with $K$, we get $\int_{\Delta^{p}}\colon \Omega^{p}_{K[\Delta^{p}]/K} \to K$.
More generally, for a $K$-algebra $A$ we can define an integration map $\int_{\Delta^{p}}\colon \Omega^{n}_{A[\Delta^{p}]/K} \to \Omega^{n-p}_{A/K}$ using the decomposition
$\Omega^{n}_{A[\Delta^{p}]/K} \cong \bigoplus_{k+l=n} \Omega^{k}_{A/K} \otimes_{K} \Omega^{l}_{K[\Delta^{p}]/K}$.

If $A$ is a $K$-dagger algebra one can show 
similarly as in \cite[Appendix]{TammeBorel} that by continuity this extends uniquely to a map
\[
\int_{\Delta^{p}}\colon \Omega^{n}_{A\<\Delta^{p}\>^{\dag}/K,\mathrm f} \to \Omega^{n-p}_{A/K, \mathrm f},
\]
where the subscript $\mathrm f$ indicates that we consider differential forms for dagger algebras (cf. \ref{mar2005}).
The analogues of \eqref{mar2006} and \eqref{mar2007} remain valid.
\end{Nr}

\begin{Nr}
Let $Y_{\dot}$ be a simplicial dagger space. A \emph{simplicial} $n$-form $\omega$ on $Y_{\dot}$ is a collection of $n$-forms $\omega_{p}$ on $Y_{p} \times \widehat\Delta^{p}_{K}$, $p\geq 0$, satisfying $(1\times\phi_{\Delta})^{*}\omega_{p} = (\phi_{X}\times 1)^{*}\omega_{q}$ on $Y_{p}\times\widehat\Delta^{q}_{K}$ for all monotone maps $\phi\colon [q]\to [p]$, $\phi_{\Delta},\phi_{X}$ denoting the induced (co)simplicial structure maps. We will often denote $\omega_{p}$ by $\omega|_{Y_{p}\times\widehat\Delta^{p}_{K}}$.

The space of all simplicial $n$-forms is denoted by $D^{n}(Y_{\dot})$. Applying the wedge product and exterior differential component-wise makes $D^{*}(Y_{\dot})$ into a commutative differential graded algebra. 

For a $K$-dagger space $X$ we write $\Omega^{*}(X)$ for the complex of global sections $\Gamma(X, \Omega^{*}_{X/K})$.
Define $\Omega^{*}(Y_{\dot})$ as the total complex  of the cosimplicial complex $[p]\mapsto \Omega^{*}(Y_{p})$.
By Dupont's Theorem \cite[Thm. 5.6]{TammeDiss}
\begin{equation}\label{eq:DefI}
I\colon D^{*}(Y_{\dot}) \to \Omega^{*}(Y_{\dot}), \quad \omega \mapsto \sum_{k} \int_{\Delta^{k}} \omega|_{Y_{k}\times\widehat\Delta^{k}_{K}}
\end{equation}
is a quasi-isomorphism.

We have a decomposition $\Omega^n_{X\times\widehat\Delta^k_{K}/K} \cong \bigoplus_{p+q=n} \Omega^{p,q}_{X\times\widehat\Delta^k_{K}/K}$ where 
$\Omega^{p,q}_{X\times\widehat\Delta^k_{K}/K} :=
\operatorname{pr}_X^*\Omega^p_{X/K} \otimes \operatorname{pr}_{\widehat\Delta^k_{K}}^*\Omega^q_{\widehat\Delta^k_{K}/K}$
and hence the filtration $\Fil_X^\dot\Omega^*(X\times\widehat\Delta^k_{K})$ with respect to the first degree. This induces
a filtration $\Fil^\dot$ on $D^{*}(Y_{\dot})$  and $I$ is a filtered quasi-isomorphism when $\Omega^{*}(Y_{\dot})$ 
carries the filtration given by $\Fil^{n}\Omega^{*}(Y_{\dot}) = \Omega^{\geq n}(Y_{\dot})$.
\end{Nr}

\begin{Nr}
Let $X$ be an affine scheme in $\Sm_{R}$ with generic fibre $X_{K}$. Recall the complexes $\Fil^{n}R\Gamma_{\dR}(X_{K}/K)$ from \eqref{eq:FilDRCompl}. Similarly as in \eqref{eq:GdCompMap}, \eqref{eq:ComplCompMap} there is a natural map
\begin{equation}\label{eq:KarCompMap1}
\Fil^{n}R\Gamma_{\dR}(X_{K}/K) \to \Gamma(X_{K}, \Gd_{Pt(X_{K}^{\dag})\sqcup Pt(X_{K})} \Omega^{\geq n}_{X_{K}/K})
\end{equation}
where the complex on the right computes the hypercohomology of $\Omega^{\geq n}_{X_{K}/K}$ on $X_{K}$. Since $X_{K}$ is affine the sheaves $\Omega^{i}_{X_{K}/K}$ are acyclic and hence the coaugmentation 
\begin{equation}\label{eq:KarCompMap2}
  \Gamma(X_{K}, \Gd_{Pt(X_{K}^{\dag})\sqcup Pt(X_{K})} \Omega^{\geq n}_{X_{K}/K}) \leftarrow \Gamma(X_{K}, \Omega^{\geq n}_{X_{K}/K}) 
\end{equation}
is a quasi-isomorphism.  The maps \eqref{eq:KarCompMap1}, \eqref{eq:KarCompMap2}, and 
\[
\Gamma(X_{K}, \Omega^{\geq n}_{X_{K}/K}) \to \Gamma(X_{K}^{\dag}, \Omega^{\geq n}_{X_{K}^{\dag}/K}) \to \Gamma(\widehat X_{K}, \Omega^{\geq n}_{\widehat X_{K}/K})
\]
give  a chain of morphisms
connecting $\Fil^{n}R\Gamma_{\dR}(X_{K}/K)$ with $\Gamma(\widehat X_{K}, \Omega^{\geq n}_{\widehat X_{K}/K})$ where the arrow \eqref{eq:KarCompMap2} pointing in the wrong direction is a quasi-isomorphism. All these morphisms are functorial in $X\in \Sm_{R}$. We denote this by
\begin{equation}\label{eq:KarCompMap3}
\Fil^{n}R\Gamma_{\dR}(X_{K}/K) \to\xleftarrow{\qis} \Gamma(\widehat X_{K}, \Omega^{\geq n}_{\widehat X_{K}/K}).
\end{equation}
If $Y_{\dot}$ is a simplicial smooth, affine $R$-scheme this and the quasi-isomorphism $I$ from \eqref{eq:DefI} give a natural chain of morphisms
\begin{equation}\label{eq:KarCompMap4}
\Fil^{n}R\Gamma_{\dR}(Y_{K,\dot}/K) \to\xleftarrow{\qis} \Fil^{n}D^{*}(\widehat Y_{K,\dot})
\end{equation}
\end{Nr}

We apply \eqref{eq:KarCompMap4} to $B_{\dot}\GL_{r,R}$ to get a map
\begin{equation}\label{eq:MapDeRhamDupont}
\Fil^{n}H^{2n}_{\dR}(B_{\dot}\GL_{r,K}/K) \to H^{2n}(D^{*}(B_{\dot}\widehat\GL_{r,K})).
\end{equation}
Similarly as in Lemma \ref{nov1204} there is an isomorphism
\begin{equation}\label{mar0801}
H^{2n}\left(\MF\big(\Fil^{n}D^{*}(B_{\dot}\widehat\GL_{r,K}) \to D^{*}(E_{\dot}\widehat\GL_{r,K})\big)\right) \xrightarrow{\cong}
H^{2n}(D^{*}(B_{\dot}\widehat\GL_{r,K}))
\end{equation}
and we define 
\begin{equation}\label{eq:DefChnKar}
\ch_{n,(r)}^{\mathrm{Kar}} \in H^{2n}\left(\MF\big(\Fil^{n}D^{*}(B_{\dot}\widehat\GL_{r,K}) \to D^{*}(E_{\dot}\widehat\GL_{r,K})\big)\right)\end{equation}
to be the unique element whose image under \eqref{mar0801} coincides with
 the image  of $\ch_{n,(r)}^{\dR}$  under \eqref{eq:MapDeRhamDupont}.

\begin{rem*}
Originally, Karoubi used Chern-Weil theory to construct the relevant characteristic classes. This is the reason for the use of the dga $D^{*}(B_{\dot}\widehat\GL_{K})$. One advantage is that it gives more explicit formulas for the relative Chern character. For instance, these are used for the comparison theorem in \cite{TammeBorel} and in the work of Choo and Snaith \cite{ChooSnaith}. It was checked in \cite[Prop. 5.13]{TammeDiss} that the present approach yields the same classes as the Chern-Weil theoretic one.
\end{rem*}

To describe Karoubi's version of the relative Chern character we first need a Lemma. Let $X=\Spec(A)$ be a regular, affine $R$-scheme. 
We consider $B_\dot\GL(A)$ as a bisimplicial set which is constant in the second direction. Using this, we define the bisimplicial set $F^\flat(A) := B_\dot\GL(A) \times_{B_\dot\GL(A^\dag\<\Delta^\dot\>^\dag)} E_\dot\GL(A^\dag\<\Delta^\dot\>^\dag)$ and similarly $F^{\flat}_{r}(A)$ for every finite level $r$. These are bisimplicial subsets of $F(A)$, respectively $F_{r}(A)$, from \eqref{nov0904} and \eqref{dez1202}.
\begin{lemma}\label{mar2003}
 The induced map on complexes $\Z F^\flat(A) \to \Z F(A)$ is a quasi-isomorphism.
\end{lemma}
\begin{proof}
By the Eilenberg-Zilber theorem it suffices to show that $\Z\diag F^{\flat}(A) \to \Z\diag F(A)$ is a quasi-isomorphism. For this it is enough to show that the map $\diag F^{\flat}(A) \to \diag F(A)$ is acyclic. By the definitions we have a pullback square
\begin{equation}\label{mar0802}\begin{split}
\xymatrix@C-0.3cm@R-0.2cm{
\diag F^{\flat}(A) \ar[r]\ar@{->>}[d] & \diag F(A) \ar@{->>}[d] \\
B_{\dot}\GL(A) \ar[r] & \diag B_{\dot}\GL(A[\Delta^{\dot}]).
}
\end{split}
\end{equation}
Since $\diag E_{\dot}\GL(A^{\dag}\<\Delta^\dot\>^\dag) \to \diag B_{\dot}\GL(A^{\dag}\<\Delta^\dot\>^\dag)$ is a Kan fibration, so are the vertical maps in \eqref{mar0802}.
It thus suffices to show that the lower horizontal map is acyclic (cf.~\cite[(4.1)]{Berrick}).

Let $(\,.\,)^{+}$ denote Quillen's plus construction. Then $\pi_{i}(B_{\dot}\GL(A)^{+})=K_{i}(A), i\geq 1$, are the Quillen K-groups of $A$.
The map $B_{\dot}\GL(A) \to B_{\dot}\GL(A)^{+}$ is acyclic.
We have a commutative diagram
\[
\xymatrix@R-0.2cm{
B_{\dot}\GL(A) \ar[r] \ar[d]_{\text{acyclic}} & \diag B_{\dot}\GL(A[\Delta^{\dot}]) \ar[d]^{\simeq}\\
B_{\dot}\GL(A)^{+} \ar[r]^-{\simeq} & \diag B_{\dot}\GL(A[\Delta^{\dot}])^{+}
}
\]
where the lower horizontal map is a weak equivalence by the homotopy invariance for K-theory of regular rings, and the right vertical map is a weak equivalence since $B_{\dot}\GL(A[\Delta^{\dot}])$ has the homotopy type of an H-space. Hence the upper horizontal map is acyclic, as desired.
\end{proof}

Karoubi's relative Chern character $\ch_{n,i}^{\mathrm{Kar}}$ is defined as the composition
\begin{align*}
K_{i}^{\rel}(X) &\to H_{i}(\diag F(A), \Z) && \text{(Hurewicz)} \notag \\
&\cong H_{i}(\diag F^\flat(A), \Z)  && \text{(Lemma \ref{mar2003})} \\
&\cong \varinjlim_{r} H_{i}(\diag F^{\flat}_{r}(A),\Z) \\
&\to H^{2n-i-1}(\Omega^{<n}(\widehat X_{K})) &&\text{(using \eqref{eq:FastKaroubisRelChern})}
\end{align*}
where the last map is constructed as follows (cf. \cite[Remark 3.6(ii)]{TammeDiss}):
%
%
%
An $i$-simplex $\sigma$ in $\diag F^\flat_{r}(A)$ defines a pair of morphisms $\widehat X_{K}\xrightarrow{\sigma_0} B_{i}\widehat\GL_{r,K}$ and $\widehat X_{K} \times \widehat\Delta^{i}_{K} \xrightarrow{\sigma_1} E_{i}\widehat\GL_{r,K}$ such that 
\[
\xymatrix{
\widehat X_{K} \times \widehat \Delta^{i}_{K} \ar[r]^-{\sigma_1} \ar[d]_{\mathrm{proj.}} & E_{i}\widehat\GL_{r,K}\ar[d]^{p}   \\
\widehat X_{K} \ar[r]^-{\sigma_0} & B_{i}\widehat\GL_{r,K} 
}
\quad\text{and hence}
\xymatrix@C+0.5cm{
& E_{i}\widehat\GL_{r,K}\times\widehat\Delta^{i}_{K} \ar[d]^{p\times\id_{\Delta^{i}}} \\
 \widehat X_{K} \times \widehat\Delta^{i}_{K} \ar[ur]^{(\sigma_1,\pr_{\Delta^{i}})}   \ar[r]^-{\sigma_0\times\id_{\Delta^{i}}} & B_{i}\widehat\GL_{r,K} \times\widehat\Delta^{i}_{K}
}
\]
commute.
We can write 
\begin{equation}\label{mar2004}
\ch_{n,(r)}^{\mathrm{Kar}} = (\omega_0, \omega_1)
\end{equation}
with $\omega_0 \in \Fil^nD^{2n}(B_\dot\widehat\GL_{r,K}), \omega_1\in D^{2n}(E_\dot\widehat{\GL}_{r,K})$ and $d\omega_1=p^*\omega_0$.
Then 
\begin{multline*}
((\sigma_0\times\id_{\Delta^i})^*\omega_0, (\sigma_1,\pr_{\Delta^i})^*\omega_1)
:=\\
((\sigma_0\times\id_{\Delta^i})^*\omega_0|_{E_{i}\widehat\GL_{r,K}\times\widehat\Delta^{i}_{K}}, (\sigma_1,\pr_{\Delta^i})^*\omega_1|_{E_{i}\widehat\GL_{r,K}\times\widehat\Delta^{i}_{K}})
\end{multline*}
is a cycle of degree $2n$ in $\MF\left(\Fil^n_{\widehat X_K}\Omega^{*}(\widehat X_{K}\times\widehat\Delta^{i}_{K}) \to \Omega^{*}(\widehat X_{K}\times\widehat\Delta^{i}_{K})\right)$.
We now integrate along $\Delta^{i}$ and use the quasi-isomorphism $\MF\big(\Omega^{\geq n}(\widehat X_{K}) \to \Omega^{*}(\widehat X_{K})\big) \xrightarrow{\sim} \Omega^{<n}(\widehat X_{K})[-1]$ induced by the projection to the second component to get the map
\[
 F^\flat_{r}(A) \ni \sigma \mapsto \int_{\Delta^i} (\sigma_{1},\pr_{\Delta^{i}})^*\omega_{1} \in \Omega^{<n}(\widehat X_{K})^{2n-i-1}.
\]
This induces a well defined homomorphism 
\begin{equation}\label{eq:FastKaroubisRelChern}
H_{i}(\diag F^\flat_{r}(A), \Z) \to 
H^{2n-i-1}(\Omega^{<n}(\widehat X_{K}))
\end{equation}
compatible for varying $r$.

To compare $\ch_{n,i}^{\mathrm{Kar}}$ with our version of the relative Chern character we need the following
\begin{lemma}\label{lemma:ComparisonMap}
Let $X$ be a smooth, affine  $R$-scheme. There is a natural map
\begin{equation}\label{eq:ComparisonMap}
H^{*}_{\rel}(X,n) \to H^{*}\left(\MF(\Omega^{\geq n}(\widehat X_{K}) \to \Omega^{*}(\widehat X_{K}))\right) \cong H^{*-1}(\Omega^{<n}(\widehat X_{K})).
\end{equation}
\end{lemma}
\begin{proof}
We have $H^{*}_{\rel}(X,n) =
\MF(\Fil^{n}R\Gamma_{\dR}(X_{K}/K) \to R\Gamma_{\dR}(\widehat X_{K}/K))$.
Since $X$ is affine we have the natural chain of morphisms $\Fil^{n}R\Gamma_{\dR}(X_{K}/K) \to\xleftarrow{\qis} \Omega^{\geq n}(\widehat X_{K})$ from \eqref{eq:KarCompMap3} and similarly $R\Gamma_{\dR}(\widehat X_{K}/K) \xleftarrow{\qis} \Omega^{*}(\widehat X_{K})$. These together induce the desired chain of maps
\begin{equation*}
\MF\left(\Fil^{n}R\Gamma_{\dR}(X_{K}/K) \to R\Gamma_{\dR}(\widehat X_{K}/K)\right) \to \xleftarrow{\qis}
\MF\left(\Omega^{\geq n}(\widehat X_{K}) \to \Omega^{*}(\widehat X_{K})\right).  \qedhere
\end{equation*}
\end{proof}

\begin{thm}\label{thm:KaroubisRelChern}
Let $X$ be a smooth, affine $R$-scheme.
The composition of the relative Chern character $\ch_{n,i}^{\rel}$ of Definition \ref{nov1601} with the comparison map \eqref{eq:ComparisonMap} coincides with Karoubi's relative Chern character $\ch_{n,i}^{\mathrm{Kar}}$.
\end{thm}
\begin{proof}
Let $X=\Spec(A)$.
Using the quasi-isomorphism $\Z F^\flat(A) \xrightarrow{\qis} \Z F(A)$ we can describe the composition
\begin{equation}\label{eq:RelChernNaive0}
K_{i}^{\rel}(X) \xrightarrow{\ch_{n,i}^{\rel}} H^{2n-i}_{\rel}(X, n) \xrightarrow{\eqref{eq:ComparisonMap}} H^{2n-i-1}(\Omega^{<n}(\widehat X_{K}))
\end{equation}
explicitly as follows: 
We have a natural map (cf. the construction in Lemma \ref{lemma:ComparisonMap})
\begin{multline*}
H^{2n}\left(\MF\big(\Fil^{n}R\Gamma_{\dR}(B_{\dot}\GL_{r,K}/K) \to R\Gamma_{\dR}(E_{\dot}\widehat\GL_{r,K})\big)\right) \to \\
H^{2n}\left(\MF\big(\Omega^{\geq n}(B_{\dot}\widehat\GL_{r,K}) \to \Omega^{*}(E_{\dot}\widehat\GL_{r,K})\big)\right),
\end{multline*}
and we denote the image of $\ch_{n,(r)}^{\rel}$ (see \eqref{nov1401}) under this map by 
\[
\widetilde{\ch}_{n,(r)}^{\rel} \in H^{2n}\left(\MF\big(\Omega^{\geq n}(B_\dot\widehat\GL_{r,K}) \to \Omega^*(E_\dot\widehat\GL_{r,K})\big)\right).
\]
As in \eqref{nov0801} we have a map 
\begin{equation}\label{nov1603}
 H_i(\Tot\Z F^\flat_{r}(A)) \to \Hom\binom{H^{2n}(\MF(\Omega^{\geq n}(B_{\dot}\widehat\GL_{r,K}) \to \Omega^{*}(E_{\dot}\widehat\GL_{r,K}))),\qquad}{\qquad H^{2n-i}(\MF(\Omega^{\geq n}(\widehat X_{K}) \to \Omega^{*}(\widehat X_{K}\times\widehat\Delta^{\dot}_{K})))}, 
\end{equation}
that we can compose with the evaluation $\mathrm{eval}_{\widetilde{\ch}_{n,(r)}^{\rel}}$ to get 
\begin{equation}\label{mar2001}
H_i(\Tot\Z F^\flat_{r}(A)) \to H^{2n-i}(\MF(\Omega^{\geq n}(\widehat X_{K}) \to \Omega^{*}(\widehat X_{K}\times\widehat\Delta^{\dot}_{K}))).
\end{equation}
Since $\widehat X_{K} \times \widehat\Delta^{p}_{K}$ is an affinoid dagger space, hence the higher cohomology of coherent sheaves on $\widehat X_{K} \times \widehat\Delta^{p}_{K}$ vanishes, the argument of Lemma \ref{nov1206} shows that the coaugmentation $\Omega^{*}(\widehat X_{K}) \to \Omega^{*}(\widehat X_{K}\times\widehat\Delta^{\dot}_{K})$ is a quasi-isomorphism. It induces an isomorphism
\begin{equation}\label{mar2002}
H^{2n-i}\left(\MF\big(\Omega^{\geq n}(\widehat X_{K}) \to \Omega^{*}(\widehat X_{K}\times\widehat\Delta^{\dot}_{K})\big)\right) \cong H^{2n-i-1}(\Omega^{<n}(\widehat X_{K})).
\end{equation}
Now \eqref{eq:RelChernNaive0} equals the composition of 
\begin{multline}
 K_{i}^{\rel}(X)=\pi_i(F(A)) \xrightarrow{\text{Hurewicz}} H_i(\Z\diag F(A)) \overset{\text{\ref{mar2003}}}{\cong} \\
 H_i(\Z\diag F^\flat(A)) \overset{\text{Eilenberg-Zilber}}{\cong}  H_i(\Tot \Z F^\flat(A)) \cong \varinjlim_{r} H_{i}(\Tot \Z F^{\flat}_{r}(A))
\end{multline}
with the map induced on the direct limit by \eqref{mar2001} and \eqref{mar2002}.

From the definitions of $\ch_{n,(r)}^\rel$ in \eqref{nov1401} and $\ch_{n,(r)}^{\mathrm{Kar}}$ in \eqref{eq:DefChnKar} it is clear that $\widetilde{\ch}_{n,(r)}^{\rel} = I(\ch_{n,(r)}^{\mathrm{Kar}})$ where we still denote by $I$ the isomorphism induced by $I$ (see \eqref{eq:DefI}) on the cohomology of the respective mapping fibres.

Thus we have to show that
\begin{equation}\label{eq:diagKar}
\xymatrix@C+0.5cm{
H_{i}(\Tot \Z F^\flat_{r}(A)) \ar[d]_-{\eqref{nov1603}} & H_{i}(\Z\diag F^\flat_{r}(A)) \ar[l]_{\text{Eilenberg-Zilber}}^{\cong} \ar[dd]^{\eqref{eq:FastKaroubisRelChern}}\\
\Hom\binom{H^{2n}(\MF(\Omega^{\geq n}(B_{\dot}\widehat\GL_{r,K}) \to \Omega^{*}(E_{\dot}\widehat\GL_{r,K}))),}{H^{2n-i}(\MF(\Omega^{\geq n}(\widehat X_{K}) \to \Omega^{*}(\widehat X_{K}\times\widehat\Delta^{\dot}_{K})))} 
	\ar[d]_{\mathrm{eval}_{I(\ch_{n}^{\mathrm{Kar}})}} \\
H^{2n-i}(\MF(\Omega^{\geq n}(\widehat X_{K}) \to \Omega^{*}(\widehat X_{K}\times\widehat\Delta^{\dot}_{K}))) \ar[r]_-{\cong}^-{\eqref{mar2002}}
& H^{2n-i-1}(\Omega^{<n}(\widehat X_{K})) 
}
\end{equation}
commutes. The following lemma shows that, on the level of complexes, \eqref{mar2002} is given by sending an element of the mapping fibre $(\eta_{0},\eta_{1})$ to the integral of $\eta_{1}$ along the simplex $\Delta^\dot$. 
\begin{lemma}\label{lemma:ExplQuInv}
A quasi-inverse of $\Omega^{*}(\widehat X_{K}) \xrightarrow{\qis} \Omega^{*}(\widehat X_{K}\times \widehat\Delta^{\dot}_{K})$ is given by integration over the standard simplices: 
\[
\Omega^{k+p}(\widehat X_{K}\times \widehat\Delta^{p}_{K})  \to  \Omega^{k}(\widehat X_{K}), \quad \omega \mapsto (-1)^{p(p-1)/2}\int_{\Delta^{p}} \omega 
\]
\end{lemma}
\begin{proof}
The integration map is obviously left inverse to the inclusion of $\Omega^{*}(\widehat X_{K})$ as the zeroth column in the double complex $\Omega^{*}(\widehat X_{K} \times \widehat\Delta^{\dot}_{K})$. So we only have to check that the integration indeed defines a morphism of complexes. This is straight forward using the following relative version of Stokes' formula
\[
\int_{\Delta^{p}} d\omega = (-1)^{p}d\int_{\Delta^{p}} \omega + \sum_{i=0}^{p} (-1)^{i} \int_{\Delta^{p-1}} (\del^{i})^{*}\omega
\]
and keeping in mind that the total differential on $\Omega^{q}(\widehat X_{K}\times\widehat\Delta^{p}_{K})$ is given by $\omega \mapsto  (-1)^{p}d\omega + \sum_{i=0}^{p}(-1)^{i}(\del^{i})^{*}\omega$.
\end{proof}
To show the commutativity of diagram \eqref{eq:diagKar}, we start with a bisimplex $\sigma \in F^\flat_{r}(A)_{p,i-p}$ so that the degree of $\sigma$ in $\Tot\Z F^\flat_{r}(A)$ is $i$. First we compute its image in 
$\Omega^{<n}(\widehat X_{K})$ going counterclockwise. As before, $\sigma$ gives a commutative diagram
\[
\xymatrix{
\widehat X_{K} \times \widehat\Delta_{K}^{i-p} \ar[d]_{\mathrm{proj.}} \ar[r]^-{\sigma_1} & E_{p}\widehat\GL_{r,K} \ar[d] \\
\widehat X_{K}\ar[r]^-{\sigma_0} & B_{p}\widehat\GL_{r,K}.
}
\]
The component of $I(\ch_{n,(r)}^{\mathrm{Kar}})$ in simplicial degree $p$ is given by $(\int_{\Delta^{p}} \omega_{0}, \int_{\Delta^{p}}\omega_{1})$ (cf. \eqref{mar2004}).
Hence the image of $\sigma$ in the lower left corner of \eqref{eq:diagKar}  is given by $(-1)^{(i-p)(i-p-1)/2}(\sigma_{0}^{*}\int_{\Delta^{p}}\omega_{0}, \sigma_{1}^{*}\int_{\Delta^{p}}\omega_{1}) \in \MF(\Omega^{\geq n}(\widehat X_{K}) \to \Omega^{*}(\widehat X_{K} \times \widehat\Delta^{i-p}_{K}))^{2n-p}$ and its image in $\Omega^{<n}(\widehat X_{K})$ by \[
\int_{\Delta^{i-p}} \sigma_{1}^{*}\int_{\Delta^{p}} \omega_{1} = \int_{\Delta^{i-p}\times\Delta^{p}} (\sigma_{1}\times\id_{\Delta^{p}})^{*}\omega_{1}.
\]
Note that there is no sign because the signs introduced via \eqref{nov1603} (cf. \eqref{dez0702}) and the integration map of Lemma \ref{lemma:ExplQuInv} cancel out.

Next we compute the image of $\sigma$ going through diagram \eqref{eq:diagKar} clockwise. To do this, we use the shuffle map which is an inverse of the Eilenberg-Zilber isomorphism (cf. \cite[8.5.4]{Weibel}):
Recall that a $(p,i-p)$ shuffle $(\mu,\nu)$ is a permutation $(\mu_{1}, \mu_{2}, \dots, \mu_{p}, \nu_{1}, \dots, \nu_{i-p})$
of $\{1, 2,\dots, i\}$ such that $\mu_{1} < \mu_{2} < \dots < \mu_{p}$ and $\nu_{1} < \dots < \nu_{i-p}$. 
It determines a map $\mu\colon [i] \to [p]$ in the category $\Delta$ given by $\mu=s^{\mu_{1}-1}\circ s^{\mu_{2}-1}\circ\cdots\circ s^{\mu_{p}-1}$ 
and similarly $\nu\colon [i] \to [i-p]$ such that $(\mu,\nu)\colon [i] \to [p]\times [i-p]$ is a non degenerate $i$-simplex of the simplicial set $\Delta^{p}\times\Delta^{i-p}$. All such simplices arise in this way. On any simplicial resp. cosimplicial object we have induced maps $\mu^{*}, \nu^{*}$, respectively $\mu_{*}, \nu_{*}$.

In particular, we have $\mu^{*}\colon E_{p}\widehat\GL_{r,K}\to E_{i}\widehat\GL_{r,K}$ resp. $B_{p}\widehat\GL_{r,K}\to B_{i}\widehat\GL_{r,K}$ and $\nu_{*}\colon \widehat\Delta^{i}_{K}\to \widehat\Delta^{i-p}_{K}$. Recall that we started with a $(p,i-p)$-simplex $\sigma$ in $F_{r}^\flat(A)$. Then $(\mu,\nu)^{*}\sigma$ is an $i$-simplex in $\diag F_{r}^{\flat}(A)$.
The shuffle map sends $\sigma$ to 
\begin{equation*}\label{nov1901}
\sum_{(\mu,\nu)} \sgn(\mu,\nu) (\mu,\nu)^{*}\sigma,
\end{equation*}
where the sum runs over all $(p,i-p)$-shuffles $(\mu,\nu)$.
Its  image under the Karoubi construction \eqref{eq:FastKaroubisRelChern} is then given by 
\[
\sum_{(\mu,\nu)} \sgn(\mu,\nu) \int_{\Delta^{i}} \big(((\mu,\nu)^{*}\sigma)_{1},\pr_{\Delta^{i}}\big)^{*} \omega_{1}|_{E_{i}\widehat\GL_{r,K}\times\widehat\Delta^{i}_{K}}.
\]
Thus, to complete the proof we have to show the equality
\begin{multline}\label{nov1903}
\int_{\Delta^{i-p}\times\Delta^{p}} (\sigma_{1}\times\id_{\Delta^{p}})^{*} \omega_{1}|_{E_{p}\widehat \GL_{K}\times\widehat\Delta^{p}_{K}}= \\
\sum_{(\mu,\nu)} \sgn(\mu,\nu) \int_{\Delta^{i}} \big(((\mu,\nu)^{*}\sigma)_{1},\pr_{\Delta^{i}}\big)^{*} \omega_{1}|_{E_{i}\widehat\GL_{K}\times\widehat\Delta^{i}_{K}}.
\end{multline}
Describing the map that appears in the right hand side of the formula more explicitly we have that for a $(p,i-p)$-shuffle $(\mu,\nu)$ the induced map $\widehat X_{K} \times \widehat\Delta^{i}_{K} \xrightarrow{((\mu,\nu)^{*}\sigma)_{1}} E_{i}\widehat\GL_{K}$ is given by
$\mu^{*}\circ \sigma_{1} \circ (\id_{\widehat X_{K}} \times \nu_{*})$.
Hence we have
a commutative diagram
\[
\xymatrix@C+2cm{
\widehat X_{K} \times \widehat\Delta^{i-p}_{K} \times \widehat\Delta^{p}_{K} \ar[r]^{\sigma_{1}\times\id_{\Delta^{p}}} & E_{p}\widehat\GL_{r,K}\times\widehat\Delta^{p}_{K}\\
\widehat X_{K} \times \widehat\Delta^{i-p}_{K} \times \widehat\Delta^{i}_{K} \ar[r]^{\sigma_{1}\times\id_{\Delta^{i}}} \ar[u]^{\id_{X} \times \id_{\Delta^{i-p}}\times \mu_{*}} & E_{p}\widehat\GL_{r,K}\times\widehat\Delta^{i}_{K} \ar[u]_{\id_{E_{p}\GL}\times \mu_{*}} \ar[d]^{\mu^{*}\times\id_{\Delta^{i}}} \\
\widehat X_{K}\times\widehat\Delta^{i}_{K} \ar[r]^{(((\mu,\nu)^{*}\sigma)_{1},\pr_{\Delta^{i}})} \ar[u]^{(\id_{X}\times \nu_{*},\pr_{\Delta^{i}})} & E_{i}\widehat\GL_{r,K}\times \widehat\Delta^{i}_{K}.
}
\]
The composition of the left two vertical arrows is $\id_{\widehat X_{K}}\times(\nu_{*}, \mu_{*})$. The map $(\nu_{*}, \mu_{*})$ is exactly the map corresponding to the shuffle $(\mu,\nu)$ in the standard decomposition of $\Delta^{i-p} \times \Delta^{p}$ into $i$-simplices. It follows (cf.~\eqref{mar2007}) that
\begin{equation}\label{nov1902}
\int_{\Delta^{i-p}\times \Delta^{p}} (?)=\sum_{(\mu,\nu)}\sgn(\mu,\nu)\int_{\Delta^{i}} (\id_{X}\times(\nu_{*},\mu_{*}))^{*}(?).
\end{equation}

Since $\omega_{1}\in D^{2n-1}(E_{\dot}\widehat\GL_{r,K})$ is a simplicial differential form we have $(\mu^{*}\times\id_{\Delta^{i}})^{*}\omega_{1}|_{E_{i}\widehat\GL_{r,K}\times\widehat\Delta_{K}^{i}} = (\id_{E_{p}\widehat\GL_{r,K}} \times \mu_{*})^{*}\omega_{1}|_{E_{p}\widehat\GL_{r,K}\times\widehat\Delta_{K}^{p}}$.
Hence we get
\[
(\id_{X}\times(\nu_{*},\mu_{*}))^{*}(\sigma_{1}\times\id_{\Delta^{p}})^{*}\omega_{1}\big|_{E_{p}\widehat\GL_{r,K}\times\Delta^{p}} 
= (((\mu,\nu)^{*}\sigma)_{1}, \pr_{\Delta^{i}})^{*}\omega_{1}\big|_{E_{i}\widehat\GL_{r,K}\times\Delta^{i}}.
\]
Setting $(?)= (\sigma_{1}\times\id_{\Delta^{p}})^{*}\omega_{1}|_{E_{p}\widehat\GL_{r,K}\times\Delta^{p}}$ in \eqref{nov1902} we get the desired equality \eqref{nov1903}.
\end{proof}

\bibliography{KaroubiSyn-revised}

\providecommand{\bysame}{\leavevmode\hbox to3em{\hrulefill}\thinspace}
\providecommand{\MR}{\relax\ifhmode\unskip\space\fi MR }
\providecommand{\MRhref}[2]{%
  \href{http://www.ams.org/mathscinet-getitem?mr=#1}{#2}
}
\providecommand{\href}[2]{#2}
\begin{thebibliography}{Tam12b}

\bibitem[Ayo]{Ayoub}
Joseph Ayoub, \emph{Motifs des vari{\'e}t{\'e}s analytiques rigides}, preprint
  available at \href{http://user.math.uzh.ch/ayoub/}{user.math.uzh.ch/ayoub/}.

\bibitem[BdJ12]{Besser-deJeu}
Amnon Besser and Rob de~Jeu, \emph{The syntomic regulator for {$K_{4}$} of
  curves}, Pacific J. Math. \textbf{260} (2012), no.~2, 305--380.

\bibitem[Be{\u\i}84]{Beilinson}
A~Be{\u\i}linson, \emph{Higher regulators and values of {$L$}-functions},
  Current problems in mathematics, {V}ol. 24, Itogi Nauki i Tekhniki, Akad.
  Nauk SSSR Vsesoyuz. Inst. Nauchn. i Tekhn. Inform., Moscow, 1984,
  pp.~181--238. \MR{MR760999 (86h:11103)}

\bibitem[Ber82]{Berrick}
A.~Jon Berrick, \emph{An approach to algebraic {$K$}-theory}, Research Notes in
  Mathematics, vol.~56, Pitman (Advanced Publishing Program), Boston, Mass.,
  1982. \MR{MR649409 (84g:18028)}

\bibitem[Ber96]{BerCohomRig}
Pierre Berthelot, \emph{Cohomologie rigide et cohomologie rigide \`a supports
  propres, {P}remi\`ere partie (version provisoire 1991)}, Pr\'epublication
  IRMAR 96-03, January 1996.

\bibitem[Ber97]{Berthelot}
\bysame, \emph{Finitude et puret\'e cohomologique en cohomologie rigide},
  Invent. Math. \textbf{128} (1997), no.~2, 329--377, With an appendix in
  English by Aise Johan de Jong. \MR{MR1440308 (98j:14023)}

\bibitem[Bes00a]{BesserColeman}
Amnon Besser, \emph{A generalization of {C}oleman's {$p$}-adic integration
  theory}, Invent. Math. \textbf{142} (2000), no.~2, 397--434. \MR{1794067
  (2001i:14032)}

\bibitem[Bes00b]{Besser}
\bysame, \emph{Syntomic regulators and {$p$}-adic integration. {I}. {R}igid
  syntomic regulators}, Israel J. Math. \textbf{120} (2000), no.~part B,
  291--334. \MR{MR1809626 (2002c:14035)}

\bibitem[Bes12]{BesserK1surface}
\bysame, \emph{On the syntomic regulator for {$K_1$} of a surface}, Israel J.
  Math. \textbf{190} (2012), 29--66. \MR{2956231}

\bibitem[Bor74]{Borel1}
Armand Borel, \emph{Stable real cohomology of arithmetic groups}, Ann. Sci.
  \'Ecole Norm. Sup. (4) \textbf{7} (1974), 235--272 (1975). \MR{0387496 (52
  \#8338)}

\bibitem[Cal85]{Calvo}
Adina Calvo, \emph{{$K$}-th\'eorie des anneaux ultram\'etriques}, C. R. Acad.
  Sci. Paris S\'er. I Math. \textbf{300} (1985), no.~14, 459--462.

\bibitem[CCM13]{Mazzari}
Bruno Chiarellotto, Alice Ciccioni, and Nicola Mazzari, \emph{Cycle classes and
  the syntomic regulator}, Algebra Number Theory \textbf{7} (2013), no.~3,
  533--566. \MR{3095220}

\bibitem[CK88]{CK}
Alain Connes and Max Karoubi, \emph{Caract\`ere multiplicatif d'un module de
  {F}redholm}, $K$-Theory \textbf{2} (1988), no.~3, 431--463.

\bibitem[CLS98]{ChiarellottoLeStum}
Bruno Chiarellotto and Bernard Le~Stum, \emph{Sur la puret\'e de la cohomologie
  cristalline}, C. R. Acad. Sci. Paris S\'er. I Math. \textbf{326} (1998),
  no.~8, 961--963. \MR{1649945 (99f:14024)}

\bibitem[CS11]{ChooSnaith}
Zacky Choo and Victor Snaith, \emph{{$p$}-adic cocycles and their regulator
  maps}, J. K-Theory \textbf{8} (2011), no.~2, 241--249. \MR{2842931
  (2012k:19006)}

\bibitem[DM12]{DegliseMazzari}
Fr\'ed\'eric D\'eglise and Nicola Mazzari, \emph{The rigid syntomic ring
  spectrum}, \href{http://arxiv.org/abs/1211.5065}{arXiv:1211.5065}, 2012.

\bibitem[Fal89]{FaltingsCrystalline}
Gerd Faltings, \emph{Crystalline cohomology and {$p$}-adic
  {G}alois-representations}, Algebraic analysis, geometry, and number theory
  ({B}altimore, {MD}, 1988), Johns Hopkins Univ. Press, Baltimore, MD, 1989,
  pp.~25--80. \MR{1463696 (98k:14025)}

\bibitem[Ger73]{Gersten}
Stephen~M. Gersten, \emph{Higher {$K$}-theory of rings}, Algebraic
  {$K$}-theory, {I}: {H}igher {$K$}-theories ({P}roc. {C}onf. {S}eattle {R}es.
  {C}enter, {B}attelle {M}emorial {I}nst., 1972), Springer, Berlin, 1973,
  pp.~3--42. Lecture Notes in Math., Vol. 341. \MR{MR0382398 (52 \#3282)}

\bibitem[GK99]{GK}
Elmar Grosse-Kl{\"o}nne, \emph{de{R}ham-{K}omologie in der rigiden {A}nalysis},
  Preprintreihe des SFB 478 -- Geometrische Strukturen in der Mathematik,
  M{\"u}nster, Heft 39, 1999.

\bibitem[GK00]{GKRigid}
\bysame, \emph{Rigid analytic spaces with overconvergent structure sheaf}, J.
  Reine Angew. Math. \textbf{519} (2000), 73--95. \MR{1739729 (2001b:14033)}

\bibitem[Gro94]{Gros}
Michel Gros, \emph{R\'egulateurs syntomiques et valeurs de fonctions
  {$L\;p$}-adiques. {II}}, Invent. Math. \textbf{115} (1994), no.~1, 61--79.
  \MR{1248079 (95f:11044)}

\bibitem[Ham00]{HamidaBorel}
Nadia Hamida, \emph{Description explicite du r\'egulateur de {B}orel}, C. R.
  Acad. Sci. Paris S\'er. I Math. \textbf{330} (2000), no.~3, 169--172.
  \MR{MR1748302 (2001a:20073)}

\bibitem[Ham06]{HamidaCR}
\bysame, \emph{Le r\'egulateur {$p$}-adique}, C. R. Math. Acad. Sci. Paris
  \textbf{342} (2006), no.~11, 807--812.

\bibitem[Har77]{Harder}
G{\"u}nter Harder, \emph{Die {K}ohomologie {$S$}-arithmetischer {G}ruppen
  \"uber {F}unktionenk\"orpern}, Invent. Math. \textbf{42} (1977), 135--175.
  \MR{0473102 (57 \#12780)}

\bibitem[HK11]{HK}
Annette Huber and Guido Kings, \emph{A {$p$}-adic analogue of the {B}orel
  regulator and the {B}loch-{K}ato exponential map}, J. Inst. Math. Jussieu
  \textbf{10} (2011), no.~1, 149--190. \MR{2749574}

\bibitem[Hub95]{Huber}
Annette Huber, \emph{Mixed motives and their realization in derived
  categories}, Lecture Notes in Mathematics, vol. 1604, Springer-Verlag,
  Berlin, 1995. \MR{MR1439046 (98d:14030)}

\bibitem[Ivo05]{Ivorra}
Florian Ivorra, \emph{R\'ealisations {$\ell$}-adique des motifs mixtes}, Ph.D.
  thesis, Universit\'e Paris VI - Pierre et Marie Curie, 2005.

\bibitem[Kar82]{KaroubiConnexions}
Max Karoubi, \emph{Connexions, courbures et classes caract\'eristiques en
  {$K$}-th\'eorie alg\'ebrique}, Current trends in algebraic topology, {P}art 1
  ({L}ondon, {O}nt., 1981), CMS Conf. Proc., vol.~2, Amer. Math. Soc.,
  Providence, R.I., 1982, pp.~19--27. \MR{686108 (84f:57013)}

\bibitem[Kar83]{KarCR}
\bysame, \emph{Homologie cyclique et r\'egulateurs en {$K$}-th\'eorie
  alg\'ebrique}, C. R. Acad. Sci. Paris S\'er. I Math. \textbf{297} (1983),
  no.~10, 557--560.

\bibitem[Kar87]{Kar87}
\bysame, \emph{Homologie cyclique et {$K$}-th\'eorie}, Ast\'erisque (1987),
  no.~149, 1--147.

\bibitem[KV71]{KV}
Max Karoubi and Orlando Villamayor, \emph{{$K$}-th\'eorie alg\'ebrique et
  {$K$}-th\'eorie topologique. {I}}, Math. Scand. \textbf{28} (1971), 265--307
  (1972). \MR{MR0313360 (47 \#1915)}

\bibitem[Mer72]{Mer}
David Meredith, \emph{Weak formal schemes}, Nagoya Math. J. \textbf{45} (1972),
  1--38. \MR{MR0330167 (48 \#8505)}

\bibitem[MW68]{MW}
Paul Monsky and Gerard Washnitzer, \emph{Formal cohomology. {I}}, Ann. of Math.
  (2) \textbf{88} (1968), 181--217. \MR{MR0248141 (40 \#1395)}

\bibitem[Niz97]{NiziolImg}
Wies{\l}awa Nizio{\l}, \emph{On the image of {$p$}-adic regulators}, Invent.
  Math. \textbf{127} (1997), no.~2, 375--400. \MR{1427624 (98a:14031)}

\bibitem[Niz01]{NiziolCrys}
\bysame, \emph{Cohomology of crystalline smooth sheaves}, Compositio Math.
  \textbf{129} (2001), no.~2, 123--147. \MR{1863299 (2003g:14027)}

\bibitem[Qui72]{QuillenK}
Daniel Quillen, \emph{On the cohomology and {$K$}-theory of the general linear
  groups over a finite field}, Ann. of Math. (2) \textbf{96} (1972), 552--586.
  \MR{MR0315016 (47 \#3565)}

\bibitem[Tam10]{TammeDiss}
Georg Tamme, \emph{The relative {C}hern character and regulators}, Ph.D.
  thesis, Universit\"at Regensburg, 2010,
  \texttt{http://epub.uni-regensburg.de/15595/}.

\bibitem[Tam12a]{TammeBorel}
\bysame, \emph{Comparison of {K}aroubi's regulator and the $p$-adic {B}orel
  regulator}, J. K-Theory \textbf{9} (2012), no.~03, 579--600.

\bibitem[Tam12b]{TammeBeil}
\bysame, \emph{Karoubi's relative {C}hern character and {B}eilinson's
  regulator}, Ann. Sci. \'Ec. Norm. Sup\'er. (4) \textbf{45} (2012), no.~4,
  601--636.

\bibitem[Tho85]{ThomasonAlgK}
Robert~W. Thomason, \emph{Algebraic {$K$}-theory and \'etale cohomology}, Ann.
  Sci. \'Ecole Norm. Sup. (4) \textbf{18} (1985), no.~3, 437--552. \MR{826102
  (87k:14016)}

\bibitem[vdPS95]{vdPS}
Marius van~der Put and Peter Schneider, \emph{Points and topologies in rigid
  geometry}, Math. Ann. \textbf{302} (1995), no.~1, 81--103. \MR{1329448
  (96k:32070)}

\bibitem[Wei89]{WeibelKH}
Charles Weibel, \emph{Homotopy algebraic {$K$}-theory}, Algebraic {$K$}-theory
  and algebraic number theory ({H}onolulu, {HI}, 1987), Contemp. Math.,
  vol.~83, Amer. Math. Soc., Providence, RI, 1989, pp.~461--488. \MR{MR991991
  (90d:18006)}

\bibitem[Wei94]{Weibel}
\bysame, \emph{An introduction to homological algebra}, Cambridge Studies in
  Advanced Mathematics, vol.~38, Cambridge University Press, Cambridge, 1994.

\bibitem[Wei97]{WeibelHodgeCyclic}
\bysame, \emph{The {H}odge filtration and cyclic homology}, $K$-Theory
  \textbf{12} (1997), no.~2, 145--164. \MR{1469140 (98h:19004)}

\end{thebibliography}
\bibliographystyle{amsalpha}

\end{document}